\documentclass[12pt,english,reqno]{amsart}

\usepackage{amsthm,amsmath,amssymb,amsfonts}
\usepackage{mathtools,mleftright,bm}
\usepackage{booktabs,float,caption}
\usepackage{setspace,diagbox}
\usepackage{yfonts,microtype,xcolor}
\usepackage{a4wide}

\usepackage[
  pdfauthor={},
  pdfkeywords={},
  pdftitle={},
  pdfcreator={},
  pdfproducer={},
  linktocpage,colorlinks,bookmarksnumbered,linkcolor=blue,
  citecolor=blue,urlcolor=blue]{hyperref}

\theoremstyle{plain}
\newtheorem{theorem}{Theorem}
\newtheorem{lemma}[theorem]{Lemma}
\newtheorem{corollary}[theorem]{Corollary}

\theoremstyle{definition}

\newtheorem{example}[theorem]{Example}
\newtheorem*{remark}{Remark}

\numberwithin{equation}{section}
\numberwithin{theorem}{section}
\numberwithin{figure}{section}
\numberwithin{table}{section}

\DeclarePairedDelimiter{\floor}{\lfloor}{\rfloor}
\DeclarePairedDelimiter{\norm}{|}{|}
\DeclarePairedDelimiter{\coeff}{[}{]}

\newcommand{\set}[1]{\left\{#1\right\}}
\newcommand{\phrase}[1]{``\textsl{#1}"}
\newcommand{\textdef}[1]{\textit{#1}}
\newcommand{\Case}[1]{\mbox{\emph{Case #1}:}}
\newcommand{\andq}{\quad \text{and} \quad}

\newcommand{\mids}{\,\mid\,}
\newcommand{\nmids}{\,\nmid\,}
\newcommand{\valueat}[1]{{\,}_{\big|\, #1}}
\newcommand{\dop}{\mathcal{D}}
\newcommand{\dif}{\Delta}

\newcommand{\ZZ}{\mathbb{Z}}
\newcommand{\QQ}{\mathbb{Q}}

\newcommand{\CC}{\mathbb{C}}

\newcommand{\ff}{\mathfrak{f}}
\newcommand{\fh}{\widehat{\mathfrak{f}}}
\newcommand{\FF}{\mathfrak{F}}
\newcommand{\FP}{\mathbf{F}\mspace{-2.5mu}}

\newcommand{\fp}{\bm{f}\mspace{-3mu}}
\newcommand{\dd}{\mathbf{d}}
\newcommand{\bb}{\mathfrak{b}}

\newcommand{\AN}{\mathbf{A}}
\newcommand{\BN}{\mathbf{B}}
\newcommand{\BH}{\widehat{\mathbf{B}}}
\newcommand{\BS}{\mathbf{B}^\diamond}
\newcommand{\BC}{\mathcal{B}}
\newcommand{\BF}{\mathfrak{B}}
\newcommand{\GN}{\mathbf{G}}
\newcommand{\LN}{\mathbf{L}}
\newcommand{\FH}{\widehat{F}}
\newcommand{\PF}{\Lambda}
\newcommand{\SF}{\mathtt{S}}

\newcommand{\pref}[1]{\texorpdfstring{\ref{#1}}{\ref*{#1}}}

\title[Faulhaber polynomials and reciprocal Bernoulli polynomials]
{Faulhaber polynomials and\\ reciprocal Bernoulli polynomials}
\author{Bernd C. Kellner}
\address{G\"ottingen, Germany}
\email{bk@bernoulli.org}
\keywords{Faulhaber polynomial, power sum, Bernoulli number and polynomial,
reciprocal and palindromic polynomial, Genocchi and Lah number, Hoppe's formula}
\subjclass[2020]{11B57 (Primary), 11B68 (Secondary)}

\begin{document}

\begin{abstract}
About four centuries ago, Johann Faulhaber developed formulas for the power sum
$1^n + 2^n + \cdots + m^n$ in terms of $m(m+1)/2$. The resulting polynomials
are called the Faulhaber polynomials. We first give a short survey of
Faulhaber's work and discuss the results of Jacobi (1834) and the less known
ones of Schr\"oder (1867), which already imply some results published afterwards.
We then show, for suitable odd integers $n$, the following properties of the
Faulhaber polynomials $F_n$. The recurrences between $F_n$, $F_{n-1}$, and
$F_{n-2}$ can be described by a certain differential operator. Furthermore, we
derive a recurrence formula for the coefficients of $F_n$ that is the
complement of a formula of Gessel and Viennot (1989). As a main result, we show
that these coefficients can be expressed and computed in different ways by
derivatives of generalized reciprocal Bernoulli polynomials, whose values can
also be interpreted as central coefficients. This new approach finally leads to
a simplified representation of the Faulhaber polynomials. As an application,
we obtain some recurrences of the Bernoulli numbers, which are induced by
symmetry properties.
\end{abstract}

\maketitle


\section{Introduction}

Define the sum-of-powers function as
\begin{equation} \label{eq:sum-def}
  S_n(m) = \sum_{\nu=0}^{m-1} \nu^n \quad (n \geq 0).
\end{equation}
It is well known that
\begin{equation} \label{eq:sum-poly}
  S_n(x) = \frac{1}{n+1}( \BN_{n+1}(x) - \BN_{n+1} ) \quad (n \geq 0)
\end{equation}
is a polynomial of degree $n+1$, where $\BN_n(x)$ denotes the $n$-th Bernoulli
polynomial and $\BN_n = \BN_n(0) \in \QQ$ is the $n$-th Bernoulli number.
The polynomials $\BN_n(x)$ can be defined by the exponential generating function
\[
  \frac{t e^{xt}}{e^t - 1} = \sum_{n=0}^{\infty} \BN_n(x) \frac{t^n}{n!}
    \quad (|t| < 2\pi)
\]
and are explicitly given by the formula
\begin{equation} \label{eq:bern-poly}
  \BN_n(x) = \sum_{k=0}^{n} \binom{n}{k} \BN_{n-k} \, x^k \quad (n \geq 0).
\end{equation}

The first few nonzero values of the Bernoulli numbers are
\begin{equation} \label{eq:bern-values}
  \BN_0 = 1,\;\; \BN_1 = -\tfrac{1}{2},\;\; \BN_2 = \tfrac{1}{6},\;\;
  \BN_4 = -\tfrac{1}{30},\;\; \BN_6 = \tfrac{1}{42},\;\; \BN_8 = -\tfrac{1}{30},
\end{equation}
while $\BN_n=0$ for odd $n \geq 3$. The nonzero values alternate in sign
such that $(-1)^{\frac{n}{2}-1} \, \BN_n > 0$ for even $n \geq 2$.

Originally, the Bernoulli numbers and formula~\eqref{eq:sum-poly} became famous
by Jacob Bernoulli's book \emph{Ars Conjectandi}~\cite{Bernoulli:1713},
which was posthumously published in 1713.

About one century before, Johann Faulhaber (1580--1635) already found formulas
for $S_n(x)$ in a quite different way, without knowledge of the Bernoulli
numbers. Instead of using Bernoulli polynomials as in~\eqref{eq:sum-poly},
the so-called Faulhaber polynomials come into play, being the subject of the
present paper. More precisely, the power sum $S_n(x)$ can be expressed in terms
of the first sums $S_1(x)$ and $S_2(x)$. Jacobi gave the first rigorous proof
of the following theorem, which originates from Faulhaber's pioneering work.

\begin{theorem}[Faulhaber--Jacobi \cite{Jacobi:1834} (1834)]
\label{thm:Faulhaber-Jacobi}
For ${n \geq 2}$, there exist the Faulhaber polynomials $F_n \in \QQ[x]$
of degree $\floor{n/2} - 1$ such that
\[
  S_n(x) = f_n(x) \, F_n( S_1(x) )
\]
with the factor
\[
  f_n(x) = \begin{cases}
    S_2(x), & \text{if $n \geq 2$ is even}, \\
    S_1(x)^2, & \text{if $n \geq 3$ is odd}.
  \end{cases}
\]
\end{theorem}

As commonly used by several authors, the polynomials $F_n(y)$ for odd
$n \geq 3$ are defined here in such a way to get rid of an extra factor $y^2$.
This coincides with the simplified degree formula of $F_n$ and,
as we shall see later, with the property $F_n(0) \neq 0$ for $n \geq 2$.
See Table~\ref{tbl:F-poly} for the first few Faulhaber polynomials.

The next section gives a historical survey of Faulhaber's work and related
properties of power sums (the impatient reader may skip this for the moment).
Subsequently, Section~\ref{sec:jac} shows some basic results of Jacobi and
Schr\"oder, while Section~\ref{sec:appr} gives a simple approach to Schr\"oder's
theorem. From Section~\ref{sec:main} onward, the main results of the paper
are presented and developed consecutively.


\section{Faulhaber's work}

\yinipar{\color{teal}F}\textls[70]{\textfrak{aulhaber}} was known as
\emph{the great arithmetician of Ulm} in his day (from German: Rechenmeister).
He published numerous books during his life, only a few of them were devoted to
power sums. According to Schneider~\cite{Schneider:1983}, Faulhaber initially
published in \emph{Newer Arithmetischer Wegweyser}~\cite{Faulhaber:1614}
formulas for the power sum $S_n(m)$ up to exponent ${n=7}$ in 1614.
Subsequently, he published in \emph{Continuatio}~\cite{Faulhaber:1617}
formulas for exponents $n=8,\ldots,11$ in 1617, and finally up to $n=17$ in
his famous book \emph{Academia Algebrae}~\cite{Faulhaber:1631} in 1631.
The formulas were worked out in terms of $m(m+1)/2$, as well as of $m$.

Note that one has to \textsl{decode} Faulhaber's original formulas laboriously
because of an old and not commonly standardized notation. In particular, the
powers of variables, say~$x^k$, were represented by symbols, also called
the \textdef{cossic} notation (cf. Cajori~\cite[p.\,107]{Cajori:1928}),
which was used in the 16th and 17th centuries. Moreover, the text was written
in early new high German and printed in Fraktur typeface. A helpful reference
is Schneider's book~\cite{Schneider:1993}, where he illustrates a very detailed
survey of Faulhaber's mathematics and historical background.

Already in 1799, K\"astner gave a detailed overview of Faulhaber's literature
in his encyclopedia of the history of mathematics,
see \cite[pp.\,29--33, pp.\,111--151]{Kaestner:1799}.
For example, one reads therein that Faulhaber gave power sums for exponents
$8, \ldots, 11$ in \emph{Continuatio} and considered sums of sums in
\emph{Academia Algebrae} (see Theorem~\ref{thm:Faulhaber-Knuth} below).

To demonstrate Faulhaber's ability, one finds in his \emph{Academia Algebrae},
e.g., a table of the binomial coefficients $\binom{n}{k}$ up to $n=18$
(\cite[Sec.\,\textfrak{E}]{Faulhaber:1631}), and the values of the
factorials~$n!$ up to $n=20$ (\cite[Sec.\,\textfrak{D}]{Faulhaber:1631}).
Note that the digitalized copies of \cite{Faulhaber:1631} contain a few
misprinted digits (someone had corrected these digits with a pencil), but
the subsequent computations are overall correct.

Regarding the factorials,
Faulhaber initially computed the $n$-th differences of the sequence
$1^n, 2^n, \ldots$ for small $n$. He observed that these differences are constant,
and these values became known as $n!$ several decades later.
As an example, one finds a table of
iterated differences of the sequence $1^4, 2^4, \ldots, 6^4$, where the last
column contains the remaining two entries of fourth differences, namely, $24 = 4!$.
By this way, he arrived at the sequence $1, 2, 6, 24, 120, \ldots$.
He claimed the \textsl{general rule} $(n+1)! = (n+1) \times n!$
without giving a rigourous proof, but supported by a further table of differences
for the exponent $n=7$. Using this rule, he computed the higher
factorials in his table. Interestingly, he called the first examples as
\textls[70]{\textfrak{General Fundament}} and the rule as
\textls[70]{\textfrak{General Proce\ss}}, which can be vaguely understood as
\textsl{basis} and \textsl{step} of an induction in a heuristic sense,
respectively. One can definitely state that Faulhaber knew the technique
of finite differences very well (cf.~Schneider~\cite[p.\,294]{Schneider:1983}
and~\cite[Sec.\,7.3.2, p.\,147\,ff.]{Schneider:1993}).

The following procedure shows how to derive the first formulas of $S_n$ for
odd exponents~$n$ in a simple way (cf.~\cite[pp.\,290--291]{Schneider:1983}).
Using the forward difference $\dif f(x) = f(x+1) - f(x)$, we have the rule
\[
  \dif \binom{x}{n} = \binom{x}{n-1} \quad (n \geq 1).
\]
If we consider the difference of powers of binomial coefficients, essentially
here of $\binom{x}{2}$, then we start initially with the relations
\begin{equation} \label{eq:binom-sqr}
  \dif \binom{x}{2} = x \andq \dif \binom{x}{2}^{\!\!2} = x^3.
\end{equation}
Thus, summing up each equation of \eqref{eq:binom-sqr} by using the telescoping
property of the differences, this yields at once that

\[
  S_1(m) = \binom{m}{2} \andq S_3(m) = \binom{m}{2}^{\!\!2} \quad (m \geq 0).
\]
In particular, one immediately observes the well-known identity
\begin{equation} \label{eq:S1-3}
  S_1(m)^2 = S_3(m),
\end{equation}
which holds as polynomials. This identity is usually attributed to Nicomachus
of Gerasa (about 60--120), where the proof can be derived from identities
of triangular numbers known at that time similar to~\eqref{eq:binom-sqr}, as well as
to Al-Kara\u{g}\={\i} (953--1029), who gave a geometric proof (see Alten et al.
\cite[Sec.\,3.3.5, p.\,173]{ANFSSW:2003}, Bachmann~\cite[pp.\,4,\,18]{Bachmann:1910},
and Cantor~\cite[Vol.\,1, Chap.\,XXI, p.\,366, Chap.\,XXVI, p.\,473]{Cantor:1880}).

The relations in \eqref{eq:binom-sqr} were known to Faulhaber,
see Schneider \cite[Sec.\,7.3.1, p.\,140\,ff.]{Schneider:1993}.
In this vein, one obtains the next examples:
\[
  \dif \binom{x}{2}^{\!\!3} = \tfrac{1}{4}( x^3 + 3 x^5 ), \quad
  \dif \binom{x}{2}^{\!\!4} = \tfrac{1}{2}( x^5 + x^7 ), \quad
  \dif \binom{x}{2}^{\!\!5} = \tfrac{1}{16}( x^5 + 10 x^7 + 5 x^9 ).
\]
Therefore,
\begin{align*}
  S_1(x)^3 &= \tfrac{1}{4}( S_3(x) + 3 S_5(x) ), \\
  S_1(x)^4 &= \tfrac{1}{2}( S_5(x) + S_7(x) ), \\
  S_1(x)^5 &= \tfrac{1}{16}( S_5(x) + 10 S_7(x) + 5 S_9(x) ).
\end{align*}

Consequently, one can find the next formulas of $S_n$ for $n=5,7,9$ as a
function of $S_1$ quite easily. One may compare the formulas for $n=9$ using
the substitutions $y = \binom{x}{2}$ and $z = x^2$ for saving power operations:
\begin{equation} \label{eq:S9}
\begin{alignedat}{2}
  S_9(x) &= \tfrac{1}{5} y^2 ( 16 y^3 - 20 y^2 + 12 y - 3 )
    &\quad &\text{(Faulhaber)} \\
  &= \tfrac{1}{20} z \left( 2 z^4 - 10 x z^3 + 15 z^3 - 14 z^2 + 10 z - 3 \right)
    &\quad &\text{(Bernoulli).}
\end{alignedat}
\end{equation}

The general formula
\[
  S_1(x)^\ell = \frac{1}{2^{\ell-1}} \sum_{\substack{\nu=1\\2 \nmids \nu}}^{\ell}
    \binom{\ell}{\nu} S_{2\ell-\nu}(x) \quad (\ell \geq 1)
\]
was given by Stern~\cite{Stern:1878} in 1878
(cf.~Bachmann~\cite[Eq.\,(60), p.\,24]{Bachmann:1910} and Carlitz~\cite{Carlitz:1974}).
In the same volume of Crelle's journal, Lampe~\cite{Lampe:1878} immediately remarked
that Stern's formula would follow as a special case from considering products of the
general form $S_j S_k S_\ell \dotsm$.

Faulhaber, being very familiar with the summation technique of figurate numbers,
respectively, binomial coefficients, was finally led to Theorem~\ref{thm:Faulhaber-Jacobi}.
As mentioned earlier, this was first proved by Jacobi~\cite{Jacobi:1834} in 1834,
who used an application of the Euler–Maclaurin summation formula.
The latter formula was independently found by Euler~\cite{Euler:1741} and
Maclaurin \cite[\S\,II.828, p.\,672\,ff.]{Maclaurin:1742} around 1735,
which instantly implies the power sum formula~\eqref{eq:sum-poly} as a trivial case.
Indeed, Euler~\cite[p.\,17]{Euler:1741} recorded formulas of $S_n(x)$ in
terms of~$x$ up to $n=16$ as an application.

So far, we have neglected the exponents $n$ in the even case.
Faulhaber~\cite{Faulhaber:1631} was aware of the following transformation rule
between power sums regarding odd exponents $n \geq 3$ and even exponents $n-1$
(cf.~Knuth~\cite[p.\,281]{Knuth:1993}, Schneider~\cite[Sec.\,7.3, p.\,140\,ff.]{Schneider:1993}):

\begin{equation} \label{eq:S-odd-even}
\begin{aligned}
  S_n(x) &= \sum_{\nu=2}^{(n+1)/2} a_\nu \, S_1(x)^\nu, \\
  S_{n-1}(x) &= \frac{2x-1}{2n} \sum_{\nu=2}^{(n+1)/2} \nu a_\nu \, S_1(x)^{\nu-1}.
\end{aligned}
\end{equation}
Not surprisingly, the extra factor $2x-1$ in the formula of $S_{n-1}(x)$ originates from
\begin{equation} \label{eq:S2-sum}
  S_2(x) = \tfrac{1}{6} x(x-1)(2x-1) = \tfrac{1}{3} S_1(x) (2x-1).
\end{equation}

Starting with the relations in \eqref{eq:binom-sqr} and \eqref{eq:S1-3},
which can be iteratively continued with higher powers, and along
with \eqref{eq:S-odd-even}, it is very plausible to believe that Faulhaber
might have possessed a \textsl{proof} of his Theorem~\ref{thm:Faulhaber-Jacobi},
only lacking in rigorous arguments of induction.

Further elementary relations among power sums are given by
Bachmann~\cite[pp.\,18--20]{Bachmann:1910}.
Carlitz~\cite{Carlitz:1962} showed that $S_n(x)$ cannot be expressed as a polynomial
in $S_\ell(x)$ for $2 \leq \ell < n$. Beardon~\cite{Beardon:1996}
considered polynomial relations between power sums and gave a simple proof of
Theorem~\ref{thm:Faulhaber-Jacobi} by induction that covers the aforementioned
arguments. For a different proof see Prasolov~\cite[Thm.\,3.5.3, p.\,118]{Prasolov:2010}.

Actually, Faulhaber developed  more advanced formulas. He also considered
$r$-fold sums, i.e., iterated sums of sums, which can be defined as follows:
\[
  \SF_{n,1}(m) = 1^n + 2^n + \cdots + m^n \andq
  \SF_{n,r+1}(m) = \sum_{\nu=1}^{m} \SF_{n,r}(\nu) \quad (r \geq 1).
\]
Faulhaber claimed that the sum $\SF_{n,r}(m)$ can be mainly expressed as a
polynomial in $m(m+r)/2$, supported by examples for $\SF_{6,4}(m)$ and
$\SF_{8,2}(m)$ among others (cf.~Schneider~\cite[Sec.\,7.3.3, p.\,153\,ff.]{Schneider:1993}).
Knuth~\cite[Thm., p.\,280]{Knuth:1993} gave a general proof of this relation,
including Theorem~\ref{thm:Faulhaber-Jacobi} as a special case as well.

\begin{theorem}[Faulhaber--Knuth \cite{Knuth:1993} (1993)]
\label{thm:Faulhaber-Knuth}
For ${n, r \geq 1}$, there exist polynomials $g_{n,r} \in \QQ[x]$ such that
$\SF_{n,r}(x) = g_{n,r}( x(x+r) ) \, \SF_{d,r}(x)$, where $d = 1$
if $n$ is odd, and $d=2$ otherwise.
\end{theorem}

Finally, it should be noted that the power sum $S_n(m)$ is defined differently
in the literature, depending on the summation whether one counts up to $m-1$ or
$m$; we call them \textsl{modern} and \textsl{old} conventions, respectively.
From \eqref{eq:sum-poly}--\eqref{eq:bern-values}, we infer that
\begin{equation} \label{eq:sum-coeff}
  S_n(x) = \frac{x^{n+1}}{n+1} - \frac{x^n}{2}  + \dotsb \quad (n \geq 1).
\end{equation}
Since we have by~\eqref{eq:sum-def} the functional equation
$S_n(x+1) = S_n(x) + x^n$ for positive integers~$x$, $S_n(x+1)$ only differs
from $S_n(x)$ by a sign change from $\BN_1 = -\frac{1}{2}$ to $+\frac{1}{2}$
in the coefficient of $x^n$ in \eqref{eq:sum-coeff}.
Under the substitution $x \mapsto x+1$, there would be also a sign change in
formulas~\eqref{eq:S-odd-even} and~\eqref{eq:S2-sum}.
However, this sign change takes no effect on the Faulhaber polynomials~$F_n$,
since we would move from $\binom{x}{2} = x(x-1)/2$ to $\binom{x+1}{2} = x(x+1)/2$
in all identities. We may use the notation $\SF_n(x) = S_n(x+1)$ to avoid
ambiguity when needed.

Faulhaber~\cite{Faulhaber:1631} and Jacobi~\cite{Jacobi:1834}, as well as
Knuth~\cite{Knuth:1993}, who refers to Faulhaber's and Jacobi's work, used the
\textsl{old} convention $\SF_n(x)$, being more natural when counting from $1$
to $x$. In contrast, the \textsl{modern} convention is more convenient when
working with relations between the power sum $S_n(x)$ and the Bernoulli
polynomials $\BN_n(x)$ via \eqref{eq:sum-poly}.


\section{Historical remarks and results}
\label{sec:jac}

The Faulhaber polynomials $F_n(y)$ can be defined differently, depending on their
relations to the power sums $S_n(x)$ and the substitution $y = q(x)$ with
a quadratic polynomial $q(x)$.

\subsection*{Jacobi}

In 1834, Jacobi~\cite[p.\,271]{Jacobi:1834} gave a list of Faulhaber
polynomials. Being compatible with Jacobi's results,
Knuth~\cite[Sec.\,7, p.\,287]{Knuth:1993} used certain coefficients
$A_k^{(m)}$ for these polynomials. Turning to the \textsl{modern} convention,
we set $u = 2 S_1(x) = x(x-1)$ and $m = (n+1)/2$ for odd $n \geq 3$.
Then one reads
\begin{equation} \label{eq:Jacobi-poly}
  S_n(x) = \frac{1}{n+1} \mleft( A_0^{(m)} u^m + \dotsb + A_{m-1}^{(m)} u \mright),
\end{equation}
where the factor $\frac{1}{n+1}$ matches~\eqref{eq:sum-poly}.
In particular, one has $A_0^{(m)} = 1$ and $A_{m-1}^{(m)} = 0$.
This gives a connection with the Bernoulli polynomials by
\[
  \BN_{n+1}(x) - \BN_{n+1}
    = u^2 \mleft( A_0^{(m)} u^{m-2} + \dotsb + A_{m-2}^{(m)} \mright).
\]
See Table~\ref{tbl:S-poly-odd} for a list of Jacobi's polynomials.
Knuth~\cite[Eq.\,($*$), p.\,288]{Knuth:1993} further showed that the
coefficients obey the recurrence
\begin{equation} \label{eq:Knuth-recurr}
  \sum_{j=0}^{k} \binom{m-j}{2k+1-2j} A_j^{(m)} = 0 \quad (1 \leq k < m),
\end{equation}
while the following recurrence goes back to Jacobi~\cite[p.\,271]{Jacobi:1834},
cf.~\cite[Eq.\,($**$), p.\,289]{Knuth:1993}, simplified by using binomial
coefficients:
\begin{equation} \label{eq:Jacobi-recurr}
  \binom{2m}{2} A_k^{(m-1)}
    = \binom{2m-2k}{2} A_k^{(m)} + \binom{m-k+1}{2} A_{k-1}^{(m)}.
\end{equation}
In Section~\ref{sec:main}, recurrence~\eqref{eq:Knuth-recurr} will return
in a modified form.

We give now some historical remarks including some curiosities. While
Bernoulli~\cite[p.\,95]{Bernoulli:1713} referred to Faulhaber in 1713, and
K\"astner~\cite{Kaestner:1799} gave an overview of Faulhaber's work in 1799,
it seems that Faulhaber had been forgotten since then for a long time until
around 1980, according to Edwards~\cite{Edwards:1982}. However, in Cantor's
summary of the history of mathematics of 1892, Faulhaber was mentioned
numerously (cf.~\cite[Vol.\,2, pp.\,683--684]{Cantor:1880}).

Jacobi's paper~\cite{Jacobi:1834} of 1834 gives no credit to Faulhaber,
and only two sparse references to Maclaurin~\cite{Maclaurin:1742} and
Poisson~\cite{Poisson:1826}, whereas results of Euler are implicitly assumed.
For example, one finds in the paper Euler's zeta value $\zeta(2) = \pi^2/6$
(today Riemann's zeta function) and certain coefficients $\alpha_m$.
Citing Euler~\cite{Euler:1740,Euler:1770} thoroughly would show that
the coefficients $\alpha_m$ are effectively connected with the Bernoulli numbers
via $\alpha_m = \norm{\BN_{2m}}/(2m)!$ (in modern notation).
Apparently, it seems that subsequent related publications in the period of
1834--1980 were influenced only by Jacobi's paper.

In the correspondence between Gauss and Schumacher, which lasted for many years,
Schumacher~\cite[No.\,1147, pp.\,298--299]{Gauss&Schumacher:1863} reported on
April 8, 1847, the following. Jacobi wrote to Schumacher with some pride
(translated analogously from German):

\begin{quotation}
\begin{center}
\phrase{Regarding power sums, the relation $\SF_1^2(x) = \SF_3(x)$ is
certainly unique.\\
However, the sums of two and more of these power sums multiplied\\
by certain coefficients may be equal to a power of a single one, e.g.,
\[
  \tfrac{1}{2}( \SF_7(x) + \SF_5(x) ) = \SF_3^2(x) = \SF_1^4(x).
\]
The fact that $\SF_1^2(x) = \SF_3(x)$ can be found already in\\
Luca di Borgo's Summa Arithmetica.}
\end{center}
\end{quotation}

The latter reference means Luca Pacioli's \emph{Summa de arithmetica geometria}
\cite{Pacioli:1494} of 1494, which is written in Italian and gives a summary
of Renaissance mathematics (cf. Alten et al. \cite[pp. 218--220]{ANFSSW:2003}).

Now we arrive at an oddity. Edwards~\cite{Edwards:1986} remarked that he
found a copy of Faulhaber's \emph{Academia Algebrae}~\cite{Faulhaber:1631} at
Cambridge University in 1981, which was surprisingly \phrase{once the property
of Jacobi}. It is not known when Jacobi owned this book or if he even read it.
The above letter was written four years before his death.

For more details of Jacobi's life and achievements in mathematics see
Dirichlet's~\cite{Dirichlet:1856} obituary for Jacobi (1804--1851),
also see O'Connor and Robertson~\cite{MacTutor:Jacobi}.

\subsection*{Schr\"oder}

In 1867, Schr\"oder~\cite{Schroeder:1867} discussed the theory of the
Euler-Maclaurin summation formula and the Bernoulli polynomials, also extending
some results of Jacobi. It seems that this reference probably did not become
well known, because Schr\"oder's treatise was published in a special volume of
a cantonal school in Z\"urich. However, one finds this reference also in
N{\o}rlund~\cite{Norlund:1924}, but without any applications.

Now we briefly show the steps of Schr\"oder's proof for his version of the
Faulhaber polynomials, where we adapt and simplify his notation slightly.
Schr\"oder~\cite[\S\,4, pp.\,15--17]{Schroeder:1867} used the polynomials
$P_n(x)$ defined by
\begin{equation} \label{eq:Schroeder-poly}
  S_n(x) = n! \, P_n(x) = \frac{1}{n+1} ( \BN_{n+1}(x) - \BN_{n+1} ) \quad (n \geq 1).
\end{equation}
It shall turn out that these polynomials are compatible with Jacobi's except
for sign changes of the coefficients.

In 1848, Raabe~\cite[Sec.\,2, p.\,17\,ff.]{Raabe:1848}
(cf.~N{\o}rlund~\cite[Eq.\,(18), p.\,128]{Norlund:1922})
showed for the Bernoulli polynomials the multiplication formula
\begin{equation} \label{eq:bern-multi}
  \BN_n( kx ) = k^{n-1} \sum_{\nu=0}^{k-1} \BN_n \mleft( x + \frac{\nu}{k} \mright)
    \quad (n \geq 0, k \geq 1).
\end{equation}
(Schr\"oder referred to both publications~\cite{Raabe:1848,Raabe:1851} (1848,1851)
of Raabe, who first introduced and used the \textdef{Jacob Bernoulli function}.)
One then obtains with $k=2$ that
\begin{equation} \label{eq:P-sym}
  P_n(x) = \frac{1}{2^n} P_n \left( 2\left( x - \frac{1}{2} \right) \right)
    - P_n \left( x - \frac{1}{2} \right)
    + \left( \frac{1}{2^n} - 2 \right) \frac{\BN_{n+1}}{(n+1)!},
\end{equation}
so the polynomial $P_n(x)$ can be expressed as a function of $x - \frac{1}{2}$.
Actually, a more careful analysis reveals that $P_{2m+1}(x)$ and
$P_{2m}(x)/(x - \frac{1}{2})$ are functions of $(x - \frac{1}{2})^2$.
See Section~\ref{sec:appr} for a simple approach. Using the substitution
$(x - \frac{1}{2})^2 = \frac{1}{4}-\xi$ for $\xi = -x(x-1)$,
Schr\"oder found expressions in terms of $\xi$ as follows.

\begin{theorem}[Schr\"oder~\cite{Schroeder:1867} (1867)]
\label{thm:Schroeder}
For $m \geq 1$ with $\xi = -x(x-1)$, we have
\begin{align*}
  P_{2m+1}(x) &= \frac{(-1)^{m+1}}{(2m+2)!}
    \sum_{k=2}^{m+1} \beta_{m-k+1}^{(m)} \, \xi^k, \\
  \frac{P_{2m}(x)}{x - \frac{1}{2}} &= \frac{(-1)^{m}}{(2m+1)!}
    \sum_{k=1}^{m} \gamma_{m-k}^{(m)} \, \xi^k,
\end{align*}
where
\begin{align*}
  \beta_k^{(m)}  &= \frac{(-1)^{k+1}}{4^k} \sum_{\nu=0}^{k}
    \binom{2m+2}{2\nu} \binom{m-\nu+1}{k-\nu} ( 4^\nu - 2 ) \BN_{2\nu}, \\
  \gamma_k^{(m)} &= \frac{(-1)^{k+1}}{4^k} \sum_{\nu=0}^{k}
    \binom{2m+1}{2\nu} \binom{m-\nu}{k-\nu} ( 4^\nu - 2 ) \BN_{2\nu}.
\end{align*}
In particular, all coefficients $\beta$ and $\gamma$ are positive rational
numbers. Moreover,
\[
  (m+1) \, \gamma_k^{(m)} = (m-k+1) \, \beta_k^{(m)}.
\]
\end{theorem}

Table~\ref{tbl:S-poly-even} shows the first few polynomials $S_n(x) = n! \, P_n(x)$
in terms of $x$ and $\xi$ for even $n \geq 2$. Schr\"oder also showed the
recurrence relation (again simplified by binomial coefficients) that
\begin{equation} \label{eq:Schroeder-recurr}
  \binom{2m-2k+2}{2} \beta_k^{(m)} = \binom{m-k+2}{2} \beta_{k-1}^{(m)}
    + \binom{2m+2}{2} \beta_k^{(m-1)}.
\end{equation}
Together with the positive values
\[
  \beta_0^{(m)} = \beta_0^{(m-1)} = 1, \quad
    \beta_{m-1}^{(m)} = \frac{1}{2} \beta_{m-2}^{(m)}
    = \binom{2m+2}{2} \norm{\BN_{2m}},
\]
it easily follows by induction that all coefficients $\beta$, and so $\gamma$,
are indeed positive rational numbers, as claimed. We now compare the
polynomials in~\eqref{eq:Jacobi-poly} and~\eqref{eq:Schroeder-poly} of Jacobi
and Schr\"oder, respectively. Considering $\xi = -u$ and the different indexing
of the coefficients~$A$ and~$\beta$, this yields the relation
\begin{equation} \label{eq:coeff-A-b}
  A_k^{(m+1)} = (-1)^{k}  \beta_k^{(m)}.
\end{equation}
By plugging $A_k^{(m+1)}$ into \eqref{eq:Schroeder-recurr}, we observe, as expected,
that Schr\"oder's recurrence~\eqref{eq:Schroeder-recurr} is equivalent
to Jacobi's recurrence~\eqref{eq:Jacobi-recurr} shifted by $m \mapsto m+1$.

\subsection*{Prouhet}

In 1859, Prouhet stated two exercises in Sturm's
\emph{Cours d'Analyse}~\cite[Note\,VI, Ex.\,7--8, pp.\,359--360]{Sturm:1859},
without any references, as follows (essentially this is
Theorem~\ref{thm:Faulhaber-Jacobi}, exercises combined and translated from French):

\begin{quotation}
\begin{center}
\phrase{If $m$ is odd, then $\SF_m = n^2(n+1)^2 \varphi(n(n+1))$.
If $m$ is even, then\\ $\SF_m = n(n+1)(2n+1) \varphi(n(n+1))$.
Show that $\varphi$ is an entire function.}
\end{center}
\end{quotation}

In 1958, van der Blij~\cite{Blij:1958} gave a proof of this combined exercise,
only referring to the 6th edition of~\cite{Sturm:1859} of 1880 and nothing else
(the reviewer on zbMath remarks that a reference to N{\o}rlund~\cite{Norlund:1924}
would be helpful). The proof is similar to that given above using Bernoulli polynomials,
the multiplication formula \eqref{eq:bern-multi}, and the substitution
$t^2 = x(x-1)+\frac{1}{4}$.

\subsection*{Glaisher}

In 1899, Glaisher~\cite{Glaisher:1899} considered several formulas for $\SF_n(x)$
in terms of $x^2+x$ (also referring to Jacobi~\cite{Jacobi:1834} in this case),
$x+\frac{1}{2}$, $x+a$, and $x^2+2ax$. Some years later, Joffe~\cite{Joffe:1915}
computed the coefficients of $\SF_n(x)$ up to $n=25$ for the usual expansions in
terms of $x$, $2x+1$, and $x^2+x$.


\section{A simple approach}
\label{sec:appr}

In this section, we compare Schr\"oder's symmetry argument in the proof of
Theorem~\ref{thm:Schroeder}, which arises from the multiplication
formula~\eqref{eq:bern-multi}, with the symmetry property of $\BN_n(x)$,
which follows from the reflection formula $\BN_n(1-x) = (-1)^n \, \BN_n(x)$
(see \eqref{eq:bern-refl}). This gives a simple approach to Theorem~\ref{thm:Schroeder}.
Define the polynomials
\begin{equation} \label{eq:bern-star}
  \BS_n(x) = \BN_n(x) + \frac{n}{2} x^{n-1} \quad (n \geq 1),
\end{equation}
where the term involving $\BN_1$ is canceled out compared to~\eqref{eq:bern-poly}.
Obviously, since $\BN_n = 0$ for odd $n \geq 3$, it follows from~\eqref{eq:bern-poly}
that these polynomials have a reflection relation around $x=0$:
\[
  \BS_n(x) = (-1)^n \, \BS_n(-x).
\]
More precisely, the polynomial $\BS_n(x)$ is an odd (respectively, even)
function, if and only if $n$ is odd (respectively, even).
Further define the functions
\[
  \BH_{n,k}(x) = k^{1-n} \, \BN_n( kx ) - \BN_n( x ) \quad (n,k \geq 1).
\]
By the multiplication formula~\eqref{eq:bern-multi}, we have the relation
\begin{equation} \label{eq:bern-frac}
  \BH_{n,k}(x) = \sum_{\nu=1}^{k-1} \BN_n \mleft( x + \frac{\nu}{k} \mright).
\end{equation}
From \eqref{eq:bern-star}, we also deduce that
\[
  \BH_{n,k}(x) = k^{1-n} \, \BS_n( kx ) - \BS_n( x ),
\]
thus the symmetry property of $\BS_n(x)$ transfers to
\[
  \BH_{n,k}(x) = (-1)^n \, \BH_{n,k}(-x).
\]

In the special case $k=2$, we easily obtain the symmetry argument of Schr\"oder
in~\eqref{eq:P-sym} using the function $\BH_{n,2}(x)$. We further need the
well-known values (cf.~N{\o}rlund~\cite[Eq.\,(22), p.\,128]{Norlund:1922}),
namely, \vspace*{-1ex}
\begin{equation} \label{eq:bern-val-half}
  \BN_n \mleft( \tfrac{1}{2} \mright) = ( 2^{1-n} - 1 ) \BN_n \quad (n \geq 0),
\end{equation}
which follow from~\eqref{eq:bern-multi} with $k=2$ and $x=0$. This allows a
combined formulation as follows.
\begin{lemma} \label{lem:bern-refl}
For $n \geq 1$, we have
\begin{equation} \label{eq:bern-half}
  \BN_n(x) = \BH_{n,2} \mleft( x - \tfrac{1}{2} \mright)
    = \sum_{\substack{\nu=0\\2 \mids \nu}}^{n} \binom{n}{\nu}
    \BN_{\nu}\mleft( \tfrac{1}{2} \mright) \,
    \mleft( x - \tfrac{1}{2} \mright)^{\!n-\nu},
\vspace*{-3pt}
\end{equation}
where $\BN_n(x)$ has a reflection relation around $x=\frac{1}{2}$.
\end{lemma}

\begin{proof}
On the one hand, we obtain by the translation formula
(see \eqref{eq:bern-trans}) that
\[
  \BN_n(x) = \BN_n \mleft( x - \tfrac{1}{2} + \tfrac{1}{2} \mright)
    = \sum_{\nu=0}^{n} \binom{n}{\nu} \BN_{\nu}\mleft( \tfrac{1}{2} \mright) \,
    \mleft( x - \tfrac{1}{2} \mright)^{\!n-\nu}.
\]
Since $\BN_\nu \mleft( \frac{1}{2} \mright) = 0$ for $\nu=1$ as well as for odd
$\nu \geq 3$, the sum runs only over even indices. This is also a consequence
of the reflection formula $\BN_\nu(\frac{1}{2}) = (-1)^\nu \, \BN_\nu(\frac{1}{2})$
without the knowledge of the values of $\BN_\nu(\frac{1}{2})$.
On the other hand, it follows from~\eqref{eq:bern-frac} that
$\BH_{n,2} \mleft( x - \frac{1}{2} \mright) = \BN_n(x)$.
\end{proof}

The one-line proof of~\eqref{eq:bern-half} trivially implies that $\BN_{2n}(x)$
and $\BN_{2n+1}(x)/(x - \frac{1}{2})$ are functions of $(x - \frac{1}{2})^2$
as Schr\"oder showed in a more difficult way by~\eqref{eq:P-sym}.
Actually, the right-hand side of~\eqref{eq:bern-half} is well known.
For example, this formula was conveniently derived by
N{\o}rlund~\cite[Eq.\,(26), p.\,138]{Norlund:1922} in 1922.
Later N{\o}rlund~\cite[pp.\,21--28]{Norlund:1924} discussed the symmetry
property around $x = \frac{1}{2}$ of the Bernoulli polynomials in more detail.

Comparing with the formulas~\eqref{eq:S9} of Faulhaber and Bernoulli for $S_9(x)$,
the corresponding polynomial in terms of the integral substitution
$\omega = (2x-1)^2$ reads
\vspace*{-3pt}
\[
  S_9(x) = \frac{1}{10 \cdot 2^{10}}
    (\omega^5 - 15 \omega^4 + 98 \omega^3
    - 310 \omega^2 + 381 \omega - 155)
\]
(cf. Glaisher~\cite{Glaisher:1899} and Joffe~\cite{Joffe:1915}, as mentioned
in the former section).

Proceeding with the Faulhaber polynomials, we only need the right-hand side
of~\eqref{eq:bern-half} below. Again, with $(x - \frac{1}{2})^2 = 2y + \frac{1}{4}$
for $y = x(x-1)/2$, one then obtains
\begin{alignat*}{3}
  \BN_n(x) &= \sum_{\substack{\nu=0\\2 \mids \nu}}^{n} \binom{n}{\nu}
    \BN_{\nu}\mleft( \tfrac{1}{2} \mright) \, \mleft( 2y + \tfrac{1}{4} \mright)^{\!(n-\nu)/2}
    &\quad &(n \geq 4 \text{\ even}),\\
  \frac{\BN_n(x)}{x - \frac{1}{2}} &= \sum_{\substack{\nu=0\\2 \mids \nu}}^{n} \binom{n}{\nu}
    \BN_{\nu}\mleft( \tfrac{1}{2} \mright) \, \mleft( 2y + \tfrac{1}{4} \mright)^{\!(n-1-\nu)/2}
    &\quad &(n \geq 3 \text{\ odd}).
\end{alignat*}

Finally, together with $n S_{n-1}(x) = \BN_n(x) - \BN_n$, this provides a similar
result as given by Theorem~\ref{thm:Schroeder}, but here by expanding the sums
in terms of~$y = \binom{x}{2}$.


\section{Main results}
\label{sec:main}

We write the Faulhaber polynomial $F_n$ with its coefficients
for $n \geq 2$ and the degree $\dd_n = \deg F_n = \floor{n/2} - 1$ as
\[
  F_n(y) = \sum_{k=0}^{\dd_n} \ff_{n,k} \, y^k
    = \sum_{k=0}^{\infty} \ff_{n,k} \, y^k,
\]
where $\ff_{n,k} = 0$ for $k > \dd_n$ to keep formulas simple.
Define the differential operators
\[
  \dop_t = \frac{d}{dt} \andq
    \theta_{t,\alpha} = 1 + \alpha \, t \, \dop_t, \quad
    \theta_t = \theta_{t,1} \quad (\alpha \geq 0).
\]
The operator $\theta$ (meaning $\theta_{t,\alpha}$ or $\theta_t$), which is
linear and commutative, preserves the degree of a polynomial~$f$ such that
$\deg \theta f = \deg f$ and keeps constant terms invariant.

\begin{theorem} \label{thm:FP-recurr-1}
The Faulhaber polynomials $F_n$ obey the following recurrences:
\begin{alignat*}{3}
  n F_{n-1}(y) &= 3 \theta_{y, \frac{1}{2}} F_n(y) &\quad &(n \geq 3 \text{\ odd}), \\
  n y^2 F_{n-1}(y) &= \tfrac{1}{6} \theta_y ( F_n(y) - F_n(0) )
    + 2y \, \theta_{y,\frac{2}{3}} F_n(y) &\quad &(n \geq 4 \text{\ even}).
\end{alignat*}
In particular, we have
\[
  F_n(0) = \ff_{n,0} = \begin{cases}
    6 \BN_n, & \text{if $n \geq 2$ is even}, \\
    2n \BN_{n-1}, & \text{if $n \geq 3$ is odd},
  \end{cases}
\]
where $(-1)^{\dd_n} F_n(0) > 0$; while
$F_n(1) = \ff_{n,0} + \dotsb + \ff_{n,\dd_n} = 1$ for $n \geq 2$.
\end{theorem}

Note that $S_n(1) = 0$ and $F_n(1) = 1$, so the sum of the coefficients of
$S_n$ and $F_n$ equals $0$ and $1$, respectively. Applying the operator
$\theta_{y, \frac{1}{2}}$ to $F_n$ immediately gives the following relation,
which also follows from~\eqref{eq:S-odd-even} and~\eqref{eq:S2-sum}.
Note that $\dd_n = \dd_{n-1}$ for odd $n \geq 3$.

\begin{corollary} \label{cor:FP-odd-even}
Let $n \geq 3$ be odd. Then we have a mapping of the coefficients between
$F_n$~and $F_{n-1}$ and vice versa, as follows:
\begin{align*}
  \ff_{n-1,k} &= \frac{3}{2n} ( k+2 ) \, \ff_{n,k}
    \quad (0 \leq k \leq \dd_n), \\
\shortintertext{respectively,}
  F_{n-1}(y) &= \frac{3}{2n} \sum_{k=0}^{\dd_n} ( k+2 ) \, \ff_{n,k} \, y^k.
\end{align*}
\end{corollary}

Therefore, it is sufficient to consider the Faulhaber polynomials $F_n$
for odd $n \geq 3$ below.

\begin{theorem} \label{thm:FP-recurr-2}
Let $n \geq 5$ be odd. Then we have the recurrence
\[
  n(n-1) \, y^2 F_{n-2}(y)
    = \tfrac{1}{2} \theta_{y,\frac{1}{2}} \theta_y ( F_n(y) - F_n(0) )
    + 6y \, \theta_{y,\frac{1}{2}} \theta_{y,\frac{2}{3}} F_n(y).
\]
Equivalently, we have for $k \geq 0$ that
\begin{equation} \label{eq:coeff-recurr-2}
  \binom{n}{2} \, \ff_{n-2,k} = \frac{1}{2} \binom{2(k+3)}{2} \, \ff_{n,k+1}
    + \frac{1}{4} \binom{k+4}{2} \, \ff_{n,k+2}.
\end{equation}
In particular, the leading coefficient has for odd $n \geq 3$ the form
\[
  \ff_{n,\dd_n} = \frac{2^{(n+1)/2}}{n+1}.
\]
\end{theorem}

Edwards~\cite{Edwards:1986} showed that the coefficients of the Faulhaber
polynomials can be calculated by inverting a certain lower triangular matrix.
However, the following recurrence between the coefficients $\ff_{n,k}$ can
be easily deduced directly, see Section~\ref{sec:proof-2}.

\begin{theorem} \label{thm:coeff-recurr}
Let $n \geq 3$ be odd. For $0 \leq \ell \leq \dd_n$, we have
\begin{equation} \label{eq:coeff-recurr}
  \binom{n}{\ell+1} \BN_{n-(\ell+1)}
    = (\ell+2) \sum_{k=0}^{\ell} \binom{k+2}{\ell-k}
    \mleft(\frac{1}{2} \mright)^{\!k+2} \ff_{n,k}.
\end{equation}
In particular, for $\ell \geq 3$, the sum runs over
$\floor{(\ell-1)/2} \leq k \leq \ell$ due to vanishing binomial coefficients.
\end{theorem}

Formula~\eqref{eq:coeff-recurr} is the complement of the following
formula~\eqref{eq:coeff-calc} of Gessel and Viennot~\cite{Gessel&Viennot:1989}
that gives the coefficients $\ff_{n,k}$ explicitly. By both formulas, one can
compute the expressions of the first few coefficients $\ff_{n,k}$ for small $k$,
see Tables~\ref{tbl:F-coeff-odd} and~\ref{tbl:F-coeff-even}.

\begin{theorem}[Gessel and Viennot \cite{Gessel&Viennot:1989} (1989)]
\label{thm:Gessel-Viennot}
Let $n \geq 3$ be odd. For $0 \leq k \leq \dd_n$, we have
\begin{equation} \label{eq:coeff-calc}
  \ff_{n,k} = (-1)^k \frac{2^{k+2}}{k+2}
    \sum_{\nu=0}^{k+1} \binom{2(k+1)-\nu}{k+1} \binom{n}{\nu} \BN_{n-\nu}.
\end{equation}
Moreover, the coefficients $\ff_{n,k} \neq 0$ alternate in sign such that
$(-1)^{\dd_n - k} \, \ff_{n,k} > 0$.
\end{theorem}

Actually, formula~\eqref{eq:coeff-calc} above is
\cite[Eq.\,(12.10), p.\,33]{Gessel&Viennot:1989} after a suitable modification
of terms and notation, notably including the vanishing terms for even indices.

Observing the terms $\binom{n}{\nu} \BN_{n-\nu}$ in~\eqref{eq:coeff-calc},
which also occur in the Bernoulli polynomials in \eqref{eq:bern-poly},
one may ask whether there is a \textsl{hidden} connection.
Indeed, we can establish a relationship in a quite different manner
using \textdef{reciprocal} Bernoulli polynomials $x^n \, \BN_n(x^{-1})$,
but defined more generally. For this purpose, we introduce the related functions
\[
  \BC_{n,k}(x) = x^k \, \BN_n(x^{-1}) \quad (n \geq 0, \, k \in \ZZ),
\]
see Sections~\ref{sec:proof-3} and~\ref{sec:sym}
for further discussions. It becomes apparent
that derivatives of $\BC_{n,k}(x)$ at $x=1$ play a central role here.
We consider the related coefficients $\fh_{n,k}$, which are defined by
\begin{equation} \label{eq:coeff-subst}
  \ff_{n,k} = (-1)^{k+1} \frac{2^{k+2}}{(k+2)!} \, \fh_{n,k}.
\end{equation}
Under this substitution, we achieve the following extended results in a
different context.

\begin{theorem} \label{thm:coeff-deriv}
For $n \geq k \geq 0$, we have
\[
  \BC_{n,2k}^{(k)}(1) = (-1)^n k! \sum_{\nu=0}^{k}
    \binom{2k-\nu}{k} \binom{n}{\nu} \BN_{n-\nu}.
\]
The numbers $\bb_{n,k} = \BC_{n,2k}^{(k)}(1)$ obey, for $n \geq 2$
and $0 \leq k \leq n-2$, the recurrence
\[
  \bb_{n,k+2} = 2 (2k+3) \, \bb_{n,k+1} + n(n-1) \, \bb_{n-2,k}.
\]
For odd $n \geq 3$, we have the relations
\[
  \fh_{n,k} = \bb_{n,k+1} \quad (0 \leq k < n) \andq
    (-1)^{\dd_n+1} \, \fh_{n,k} > 0 \quad (0 \leq k \leq \dd_n).
\]
\end{theorem}

See Table~\ref{tbl:bb-coeff} for the first few values of $\bb_{n,k}$ and their
different patterns for odd and even $n$. Since $\fh_{n,k} = 0$ for $k > \dd_n$
and $\bb_{1,1} = 0$, we obtain the following recurrences of the Bernoulli numbers
in a quite simple form. One may compare this result with the lengthy formula
given by Gessel and Viennot \cite[Eq.\,below\,(12.10), p.~33]{Gessel&Viennot:1989}.

\begin{corollary} \label{cor:bern-recurr}
Let $n \geq 1$ be odd. For $(n+1)/2 \leq k \leq n$, we have
\[
  \sum_{\nu=0}^{k} \binom{2k-\nu}{k} \binom{n}{\nu} \BN_{n-\nu} = 0.
\]
\end{corollary}

It will turn out later that the vanishing of $\BC_{n,2k}^{(k)}(1)$, which induces
Corollary~\ref{cor:bern-recurr}, follows necessarily from a symmetry relation,
see Section~\ref{sec:sym}. Finally, along with Corollary~\ref{cor:FP-odd-even},
the coefficients $\ff_{n,k}$ can be represented as follows.

\begin{corollary} \label{cor:coeff-deriv}
Let $n \geq 3$ be odd. For $0 \leq k < n$, we have
\begin{align*}
  \ff_{n,k} &= (-1)^{k+1} \frac{2^{k+2}}{(k+2)!} \, \bb_{n,k+1}
\shortintertext{and}
  \ff_{n-1,k} &= (-1)^{k+1} \frac{3}{n} \frac{2^{k+1}}{(k+1)!} \, \bb_{n,k+1}.
\shortintertext{Equivalently, a relation between derivatives is given by}
  F_n^{(k)}(0) &= (-1)^{k+1} \frac{2^{k+2}}{(k+1)(k+2)} \,
    \BC_{n,2(k+1)}^{(k+1)}(1)
\shortintertext{and}
  F_{n-1}^{(k)}(0) &= (-1)^{k+1} \frac{3}{n} \frac{2^{k+1}}{k+1} \,
    \BC_{n,2(k+1)}^{(k+1)}(1).
\end{align*}
Furthermore, we have for $n \geq 2$ that $(-1)^{\dd_n - k} \, \ff_{n,k} > 0$
for $0 \leq k \leq \dd_n$.
\end{corollary}

See Example~\ref{exp:coeff} for different ways to compute $\ff_{n,k}$ for
parameters $n=13$ and $k=4$; see also Figure~\ref{fig:graph} for a graph of the
corresponding function $\BC_{13,10}^{(5)}(x)$ and the phenomenon that one of
its zeros is located very close to $x=1$. Roughly speaking, this considered zero
moves to $x=1$ regarding $\BC_{n,2k}^{(k)}(x)$ for increasing $k$.

The values of $\BC_{n,2k}^{(k)}(1)$ can be expressed also in an alternative way.
For this reason, we introduce the numbers
\[
  \BF_{n,k} = \sum_{\nu = k}^{n} \binom{n}{\nu} \binom{\nu}{k} \BN_\nu
    \quad (0 \leq k \leq n),
\]
which obey the reflection relation
\[
  \BF_{n,k} = (-1)^n \, \BF_{n,n-k}.
\]
The properties are discussed in Section~\ref{sec:sym}.
See Table~\ref{tbl:BF-coeff} for the first few values of $\BF_{n,k}$.

\begin{theorem} \label{thm:coeff-deriv-2}
For $n \geq k \geq 0$, we have
\[
  \BC_{n,2k}^{(k)}(1) = k! \sum_{\nu=0}^{k}
    \binom{2k-n}{k-\nu} \BF_{n,\nu}.
\]
\end{theorem}

Similar to Corollary~\ref{cor:bern-recurr}, we also obtain a recurrence
relation as follows.

\begin{corollary} \label{cor:bern-recurr-2}
Let $n \geq 1$ be odd. For $(n+1)/2 \leq k \leq n$, we have
\[
  \sum_{\nu=0}^{k} \binom{2k-n}{k-\nu} \BF_{n,\nu} = 0.
\]
\end{corollary}

Comparing Theorems~\ref{thm:coeff-deriv} and~\ref{thm:coeff-deriv-2},
it easily follows that the central coefficients obey for $n \geq 0$ the relation
(see Tables~\ref{tbl:bb-coeff} and~\ref{tbl:BF-coeff}) that
\[
  \bb_{2n,n} = n! \, \BF_{2n,n}.
\]
Furthermore, $\bb_{n,0} = \BF_{n,0} = (-1)^n \, \BN_n$.

For the sake of completeness, we revisit recurrence~\eqref{eq:Knuth-recurr} of
Knuth and adapt the formula for use with the coefficients $\ff_{n,k}$.
This can be accomplished by using the substitutions
\begin{equation} \label{eq:subst}
  m = \dd_n + 2 = \frac{n+1}{2} \andq
  \frac{2^{m-k}}{n+1} A_k^{(m)} = \ff_{n,\dd_n-k}.
\end{equation}
After some rearranging of terms and simplifications, we derive Knuth's
recurrence for $\ff_{n,k}$ with different indexing in the form below that
supplements Theorem~\ref{thm:coeff-recurr}. Note that Knuth did not mention
that his recurrence could be reduced as follows.

\begin{theorem}[Knuth~\cite{Knuth:1993} (1993)]
Let $n \geq 5$ be odd. For $0 \leq \ell < \dd_n$, we have
\[
  \sum_{k=\ell}^{\dd_n} \binom{k+2}{2(k-\ell) + 1}
    \mleft(\frac{1}{2} \mright)^{\!k-\ell} \ff_{n,k} = 0.
\]
The sum runs over $\ell \leq k \leq \min(2\ell+1,\dd_n)$
due to vanishing binomial coefficients.
\end{theorem}

Comparing the coefficients $\beta_k^{(m)}$ and $\ff_{n,k}$, one observes that
their explicit formulas in Theorems~\ref{thm:Schroeder} and~\ref{thm:Gessel-Viennot},
respectively, have opposite indexing and differ in terms. Schr\"oder's arguments
are simple and straightforward in showing that the coefficients $\beta_k^{(m)}$
are positive. Relations~\eqref{eq:coeff-A-b} and~\eqref{eq:subst} then imply
that the coefficients~$\ff_{n,k}$ must be nonzero and alternate in sign.
The latter properties were shown by Gessel and Viennot in a different and more
complicated way, giving a combinatorial interpretation.

However, the proofs of our theorems are self-contained and independent of
Theorems~\ref{thm:Schroeder} and~\ref{thm:Gessel-Viennot} by using other
approaches. Finally, one may show by \eqref{eq:subst} that both
recurrences~\eqref{eq:Jacobi-recurr} and \eqref{eq:Schroeder-recurr} of Jacobi
and Schr\"oder, respectively, are equivalent to \eqref{eq:coeff-recurr-2} of
Theorem~\ref{thm:FP-recurr-2}. The details are left to the reader.

The rest of the paper is organized as follows. The next section is devoted to
preliminaries. Section~\ref{sec:proof-1} contains the proofs of
Theorems~\ref{thm:FP-recurr-1} and~\ref{thm:FP-recurr-2}.
In Section~\ref{sec:proof-2}, we use Hoppe's formula to derive identities
involving derivatives of composite functions, in addition revealing a
connection with the Lah numbers. Furthermore, Hoppe's formula leads to a proof
of Theorem~\ref{thm:coeff-recurr}. In Section~\ref{sec:proof-3}, we study
properties of the reciprocal Bernoulli polynomials~$\BC_{n,k}(x)$. This results
in a proof of Theorem~\ref{thm:coeff-deriv}. Section~\ref{sec:sym}
shows some properties of the numbers~$\BF_{n,k}$, which also involve the
Genocchi numbers. The related polynomials $\BF_{n,k}(x)$ lead to a proof of
Theorem~\ref{thm:coeff-deriv-2}. Finally, it is shown that the \textdef{central}
coefficients of~$\BF_{n,k}(x)$ are connected with the coefficients of the
Faulhaber polynomials by a combined and simplified relation, see Theorem
\ref{thm:BN-FP-poly}. The final section, Section~\ref{sec:con},
contains the conclusion.


\section{Preliminaries}

The following notation and relations, which we use henceforth, can be found,
e.g., in the books~\cite{Comtet:1974,Graham&others:1994,Prasolov:2010}.
Define the falling factorial as
\[
  (x)_0 = 1, \quad (x)_n = x(x-1) \cdots (x-n+1) \quad (n \geq 1).
\]
We have the well-known relations
\begin{equation} \label{eq:binom-neg}
  \frac{(x)_n}{n!} = \binom{x}{n} \andq \binom{-x}{n} = (-1)^n \binom{x + n-1}{n}.
\end{equation}

The forward difference operator $\dif$ and its powers are given in general by
\[
  \dif^n f(x) = \sum_{\nu=0}^n  \binom{n}{\nu} (-1)^{n-\nu} f(x+\nu) \quad (n \geq 0).
\]
In particular, we have the rule
\[
  \dif^k \binom{x}{n} = \binom{x}{n-k} \quad (n \geq k \geq 0).
\]
We use the expression, for example, $\dif^n f(x+y) \valueat{x = 1}$\vspace*{-1.3ex}
to indicate the variable and an initial value when needed.

The Chu–Vandermonde identity is given by
\[
  \sum_{\nu=0}^{n} \binom{x}{\nu} \binom{y}{n-\nu} = \binom{x + y}{n}.
\]
The above identity can be rewritten by~\eqref{eq:binom-neg} as an alternating sum such that
\begin{equation} \label{eq:Chu-Vand-alt}
  \sum_{\nu=0}^{n} (-1)^\nu \binom{x + \nu-1}{\nu} \binom{y}{n-\nu} = \binom{y - x}{n}.
\end{equation}

The Bernoulli polynomials satisfy as an Appell sequence (see~\cite{Appell:1880})
the equivalent properties:
\begin{alignat}{2}
  \BN'_n(x) &= n \, \BN_{n-1}(x) &\quad &(n \geq 1), \nonumber \\
  \BN_n^{(k)}(x) &= (n)_k \, \BN_{n-k}(x) &\quad &(n \geq k \geq 1),
    \label{eq:bern-deriv} \\
  \BN_n(x+y) &= \sum_{k=0}^{n} \binom{n}{k} \BN_{n-k}(x) \, y^k &\quad &(n \geq 0).
    \label{eq:bern-trans} \\
\shortintertext{Furthermore, there are the difference and reflection relations:}
  \dif \BN_n(x) &= n x^{n-1} &\quad &(n \geq 1), \label{eq:bern-diff} \\
  \BN_n(1-x) &= (-1)^n \, \BN_n(x) &\quad &(n \geq 0). \label{eq:bern-refl}
\end{alignat}
Similarly, we deduce from~\eqref{eq:sum-poly} the following relations
for power sums:
\begin{equation} \label{eq:sum-deriv}
\begin{alignedat}{2}
  S'_n(x) &= \BN_n(x) = n S_{n-1}(x) + \BN_n &\quad &(n \geq 1), \\
  S_n^{(k+1)}(x) &= (n)_k \, \BN_{n-k}(x) &\quad &(n \geq k \geq 0).
\end{alignedat}
\end{equation}


\section{Proofs of Theorems \pref{thm:FP-recurr-1} and \pref{thm:FP-recurr-2}}
\label{sec:proof-1}

Here we derive recurrences between Faulhaber polynomials $F_n$, $F_{n-1}$, and
$F_{n-2}$. For this purpose, we shall need the following identities.

\begin{lemma} \label{lem:S12-ident}
Let $y = S_1(x) = \binom{x}{2}$ and regard derivatives with respect to~$x$.
Then we have
\[
  y' = x - \tfrac{1}{2}, \quad {y'}^2 = 2y + \tfrac{1}{4}, \quad
  S_2(x) = \tfrac{2}{3} y y', \quad S'_2(x) = 2 y + \tfrac{1}{6}.
\]
\end{lemma}

\begin{proof}
This follows from \eqref{eq:S2-sum} and $(x - \frac{1}{2})^2 = 2y + \frac{1}{4}$,
as used before.
\end{proof}

Recall Theorem~\ref{thm:Faulhaber-Jacobi} and the operator~$\theta$.

\begin{proof}[Proof of Theorem~\ref{thm:FP-recurr-1}]
Let $y = S_1(x) = \binom{x}{2}$ and $n \geq 2$.
By Theorem~\ref{thm:Faulhaber-Jacobi}, we have
\[
  S_n(x) = f_n(x) F_n(y).
\]
Due to the property $S_n(2) = f_n(2) = 1$, it follows for $x=2$, and hence $y=1$, that
\[
  F_n(1) = \ff_{n,0} + \dotsb + \ff_{n,\dd_n} = 1.
\]
Since $y \to 0$ as $x \to 0$, we infer for even $n \geq 2$ that
\begin{equation}\label{eq:F-0-even}
  F_n(0) = \lim_{y \to 0} F_n( y ) = \lim_{x \to 0} \frac{S_n(x)}{S_2(x)}
    = \lim_{x \to 0} \frac{\BN_n(x)}{\BN_2(x)} = \frac{\BN_n}{\BN_2}= 6 \BN_n,
\end{equation}
where we have used L'H\^opital's rule together with \eqref{eq:bern-values} and
\eqref{eq:sum-deriv}. Using \eqref{eq:sum-deriv} again, we compute the
derivative of $S_n$ for $n \geq 3$ in two different ways such that
\begin{equation} \label{eq:S-deriv-2}
\begin{aligned}
  S'_n(x) &= f'_n(x) F_n(y) + f_n(x) y' F'_n(y) \\
    &= n S_{n-1}(x) + \BN_n = n f_{n-1}(x) F_{n-1}(y) + \BN_n.
\end{aligned}
\end{equation}
We now divide the proof into two cases and use Lemma~\ref{lem:S12-ident}
implicitly.

\Case{$n \geq 3$ odd} We have the setting:
\[
  f_n(x) = y^2, \quad f'_n(x) = 2 y y', \quad
    f_{n-1}(x) = \tfrac{2}{3} y y', \quad \BN_n = 0.
\]
Thus, \eqref{eq:S-deriv-2} turns into
\[
  2 y y' F_n(y) + y^2 y' F'_n(y) = \tfrac{2}{3} n y y' F_{n-1}(y),
\]
which reduces to
\[
  n F_{n-1}(y) = 3 \mleft( F_n(y) + \tfrac{1}{2} y F'_n(y) \mright)
   = 3 \theta_{y, \frac{1}{2}} F_n(y),
\]
as desired. Additionally, we obtain for $y = 0$ by using~\eqref{eq:F-0-even} that
\[
  F_n(0) = \frac{n}{3} F_{n-1}(0) = 2n \, \BN_{n-1}.
\]

\Case{$n \geq 4$ even} We have the setting:
\[
  f_n(x) = \tfrac{2}{3} y y', \quad f'_n(x) = 2 y + \tfrac{1}{6}, \quad
    f_{n-1}(x) = y^2, \quad \BN_n = \tfrac{1}{6} F_n(0).
\]
From~\eqref{eq:S-deriv-2}, we deduce that
\[
  \mleft( 2y + \tfrac{1}{6} \mright) F_n(y) + \tfrac{2}{3} y {y'}^2 F'_n(y)
    = n y^2 F_{n-1}(y) + \tfrac{1}{6} F_n(0).
\]
This simplifies to
\[
  n y^2 F_{n-1}(y) = \tfrac{1}{6} \bigl( F_n(y) - F_n(0) + y F'_n(y) \bigr)
    + 2y \mleft( F_n(y) + \tfrac{2}{3} y F'_n(y) \mright).
\]
Since $\theta_y F_n(0) = F_n(0)$, we can finally write
\[
  n y^2 F_{n-1}(y) = \tfrac{1}{6} \theta_y \bigl( F_n(y) - F_n(0) \bigr)
    + 2y \, \theta_{y,\frac{2}{3}} F_n(y).
\]

At the end, it remains to show that $(-1)^{\dd_n} F_n(0) > 0$ for $n \geq 2$.
Since $\dd_n = \floor{n/2} - 1$, this easily follows from the property
$(-1)^{\frac{n}{2}-1} \BN_n > 0$ for even $n \geq 2$, completing the proof.
\end{proof}

\begin{proof}[Proof of Theorem~\ref{thm:FP-recurr-2}]
Let $n \geq 5$ be odd. By Theorem~\ref{thm:FP-recurr-1}, we have for even $n-1$ that
\begin{align}
  (n-1) y^2 F_{n-2}(y) &= \tfrac{1}{6} \theta_y \bigl( F_{n-1}(y) - F_{n-1}(0) \bigr)
    + 2y \, \theta_{y,\frac{2}{3}} F_{n-1}(y), \label{eq:rel-1} \\
  n F_{n-1}(y) &= 3 \theta_{y, \frac{1}{2}} F_n(y), \label{eq:rel-2} \\
  n F_{n-1}(0) &= 3 F_n(0). \label{eq:rel-3} \\
\shortintertext{Recall that the operator $\theta$ is linear and commutative.
By multiplying~\eqref{eq:rel-1} by $n$ and plugging~\eqref{eq:rel-2} and~\eqref{eq:rel-3}
into \eqref{eq:rel-1}, we get}
  n(n-1) \, y^2 F_{n-2}(y) &= \tfrac{1}{2} \theta_{y,\frac{1}{2}} \theta_y \bigl( F_n(y) - F_n(0) \bigr)
    + 6y \, \theta_{y,\frac{1}{2}} \theta_{y,\frac{2}{3}} F_n(y), \nonumber
\end{align}
showing the claimed relation. Applying the operators gives
\begin{align*}
  & n(n-1) \sum_{k=0}^{\dd_n-1} \ff_{n-2,k} \, y^{k+2} \\
  &\quad = \frac{1}{2} \sum_{k=1}^{\dd_n} \left( 1 + \frac{k}{2} \right)
    (1 + k) \, \ff_{n,k} \, y^{k}
    + 6 \sum_{k=0}^{\dd_n} \left( 1 + \frac{k}{2} \right)
    \left( 1 + \frac{2}{3} k \right) \ff_{n,k} \, y^{k+1}.
\end{align*}
Note that $\ff_{n,k} = 0$ for $k > \dd_n$.
By comparing the coefficients, we get for $k \geq 0$ that
\[
  n(n-1) \, \ff_{n-2,k} = \frac{1}{2} \left( 1 + \frac{k+2}{2} \right) (3 + k) \, \ff_{n,k+2}
    + 6 \left( 1 + \frac{k+1}{2} \right) \left( 1 + \frac{2}{3} (k+1) \right) \ff_{n,k+1},
\]
which finally turns into
\[
  \binom{n}{2} \, \ff_{n-2,k} = \frac{1}{2} \binom{2(k+3)}{2} \, \ff_{n,k+1}
    + \frac{1}{4} \binom{k+4}{2} \, \ff_{n,k+2}.
\]

For $k = \dd_n - 1 = (n-5)/2$ and noting that $\ff_{n,k+2} = 0$,
we obtain the plain recurrence
\begin{equation} \label{eq:coeff-lead}
  \ff_{n-2,\dd_n - 1} = \frac{1}{2} \frac{n+1}{n-1} \, \ff_{n,\dd_n}.
\end{equation}
Since $\ff_{3,0} = F_3(0) = 6 \BN_2 = 1$ by Theorem~\ref{thm:FP-recurr-1}
and~\eqref{eq:bern-values}, we inductively infer by~\eqref{eq:coeff-lead} that
$\ff_{n,\dd_n} = 2^{(n+1)/2}/(n+1)$ for odd $n \geq 3$.

Alternatively, one can find the coefficient $\ff_{n,\dd_n}$ by comparing the
leading coefficients in $S_n(x) = y^2 F_n(y)$ with $y = x(x-1)/2$
using~\eqref{eq:sum-coeff}. This proves the theorem.
\end{proof}


\section{Hoppe's formula}
\label{sec:proof-2}

Let $g, h \in C^\infty$ be smooth functions on $\CC$. In this section, we are
interested in evaluating the derivatives of a composite function $g(h(t))$.
Normally, this can be accomplished by applying the famous formula of Fa\`a di
Bruno~\cite{Bruno:1857} of 1855 (for a survey and its \textsl{curious} history
see Johnson~\cite{Johnson:2002}). However, there is a useful variant of this
formula published by Hoppe~\cite{Hoppe:1846} a few years before, which we will
apply for our purpose.

Define the signed Lah numbers for $n \geq k \geq 1$ by
\begin{equation} \label{eq:Lah-def}
  \LN_{n,k} = (-1)^n \frac{n!}{k!} \binom{n-1}{k-1},
\end{equation}
which were introduced by Lah~\cite{Lah:1954} in 1954
(cf.~Comtet~\cite[Ex.\,2, p.\,156]{Comtet:1974}).

\begin{theorem}[Hoppe~\cite{Hoppe:1846} (1846)] \label{thm:Hoppe}
For $n \geq 1$, we have
\[
  \dop_t^n \, g(h(t)) = \sum_{k=1}^{n} \frac{g^{(k)}(h(t))}{k!} \, \psi_{n,k}(h(t)),
\]
where
\[
  \psi_{n,k}(h(t)) = \sum_{j=1}^{k} \binom{k}{j} (-1)^{k-j} h(t)^{k-j} \, \dop_t^n \, h(t)^j.
\]
\end{theorem}

As a first application, we obtain the following lemma that we shall need later on
(cf.~Comtet \cite[Ex.\,7, p.\,157]{Comtet:1974}, given as an exercise for the
use of Fa\`a di Bruno's formula).

\begin{lemma} \label{lem:deriv-recip}
For $n \geq 1$, we have
\[
  \dop_t^n \, g(t^{-1}) = \sum_{k=1}^{n} \LN_{n,k} \, t^{-(n+k)} g^{(k)}(t^{-1}).
\]
\end{lemma}

\begin{proof}
Let $n \geq 1$. Note that $\dop_t^n \, t^{-j} = (-j)_n \, t^{-(j+n)}$ and
$(-j)_n/n! = \binom{-j}{n} = (-1)^n \binom{j + n-1}{n}$ for $j \geq 1$.
We further infer for $1 \leq k \leq n$ that
\[
  \sum_{j=0}^{k} \binom{k}{j} (-1)^{k-j} \binom{j + n-1}{n}
    = \dif^k \binom{x}{n} \valueat{x = n-1} = \binom{n-1}{n-k} = \binom{n-1}{k-1}.
\]
In the latter sum one can omit the index $j=0$, since $\binom{j + n-1}{n} = 0$
in that case. Hence, we obtain
\vspace*{-5pt}
\begin{align*}
  \psi_{n,k}(t^{-1}) &= \sum_{j=1}^{k} \binom{k}{j} (-1)^{k-j} t^{-(k-j)} \, (-j)_n \, t^{-(j+n)} \\
    &= (-1)^n \, n! \, t^{-(n+k)} \sum_{j=1}^{k} \binom{k}{j} (-1)^{k-j} \binom{j + n-1}{n} \\
    &= (-1)^n \, n! \, t^{-(n+k)} \binom{n-1}{k-1}.
\end{align*}
Applying Theorem~\ref{thm:Hoppe} and using~\eqref{eq:Lah-def} finally yield the result.
\end{proof}

\begin{lemma} \label{lem:deriv-power}
Let $k, n \geq 1$ and $g(t) = (t(t-1))^k$. Then we have
\[
  \dop_t^n \, g(t) \valueat{t=0} =
  \begin{cases}
    (-1)^n n! \binom{k}{n-k}, & \text{if } k \leq n \leq 2k, \\
    0, & \text{otherwise}.
  \end{cases}
\]
\end{lemma}

\begin{proof}
Since
\begin{equation} \label{eq:g-pow-t}
  g(t) = \sum_{j=0}^{k} \binom{k}{j} (-1)^{k-j} t^{k+j},
\end{equation}
there exists a term with $t^n$ in~\eqref{eq:g-pow-t} if and only if
$k \leq n \leq 2k$; otherwise, $\dop_t^n \, g(t)$ vanishes at $t=0$.
Now assume that $n \in \set{k,\ldots,2k}$.
For the index $n = k+j$, we finally infer that
\[
  \dop_t^n \, g(t) \valueat{t=0} = (-1)^n n! \binom{k}{n-k}. \qedhere
\]
\end{proof}

With the help of Hoppe's theorem, we achieve a short proof of the recurrence
of the coefficients $\ff_{n,k}$ of $F_n$ as follows.

\begin{proof}[Proof of Theorem~\ref{thm:coeff-recurr}]
Let $n \geq 3$ be odd and $0 \leq \ell \leq \dd_n = (n-3)/2$.
Set $\FH_n(y) = y^2 \, F_n(y)$, wherein the coefficients $\ff_{n,k}$ of $F_n$
are shifted. We then have by Theorem~\ref{thm:Faulhaber-Jacobi} that
\begin{equation} \label{eq:S-F}
  S_n(x) = \FH_n(g(x)),
\end{equation}
where $g(x) = \binom{x}{2}$.
We evaluate the ($\ell+2$)-th derivative of the above equation at $x=0$.
On the left-hand side of~\eqref{eq:S-F}, we get by~\eqref{eq:sum-deriv} that
\[
  \dop_x^{(\ell+2)} \, S_n(x) \valueat{x=0} = (n)_{\ell+1} \, \BN_{n-(\ell+1)}(x)
    \valueat{x=0} = (n)_{\ell+1} \, \BN_{n-(\ell+1)}.
\]
On the right-hand side of~\eqref{eq:S-F}, we deduce by means of
Theorem~\ref{thm:Hoppe} with some coefficients $\lambda_{\ell,k}$ that
\[
  \dop_x^{(\ell+2)} \FH_n(g(x)) \valueat{x=0}
    = \sum_{k=0}^{\ell+2} \frac{\FH_n^{(k)}(0)}{k!} \, \lambda_{\ell+2,k}
    = \sum_{k=0}^{\ell} \ff_{n,k} \, \lambda_{\ell+2,k+2}.
\]
One finds with Lemma~\ref{lem:deriv-power} and $g(0) = 0$ that
\[
  \lambda_{\ell+2,k+2} = \psi_{\ell+2,k+2}(g(x)) \valueat{x=0}
    = \dop_x^{(\ell+2)} g(x)^{k+2} \! \valueat{x = 0}
    = (-1)^\ell \frac{(\ell+2)!}{2^{k+2}} \binom{k+2}{\ell-k},
\]
since $\ell \geq k$. If $2k+2 < \ell$, then the latter binomial coefficient
vanishes, being compatible with Lemma~\ref{lem:deriv-power}.
Both parts imply that
\[
  \binom{n}{\ell+1} \BN_{n-(\ell+1)} = (-1)^\ell (\ell+2)
    \sum_{k=0}^{\ell} \binom{k+2}{\ell-k} \mleft(\frac{1}{2} \mright)^{\!k+2} \ff_{n,k}.
\]
If $\ell$ is odd, then $\dd_n \geq 1$ and $n \geq 5$. Thus, $n-(\ell+1) \geq 3$
and $\BN_{n-(\ell+1)} = 0$, so we can neglect the factor $(-1)^\ell$ to derive
\eqref{eq:coeff-recurr}.
Finally, the condition $2k+2 \geq \ell$ yields $k \geq \max(0, \ell/2-1)$
for the needed indices in the sum of \eqref{eq:coeff-recurr}.
For $\ell \geq 3$, this leads to $\floor{(\ell-1)/2} \leq k \leq \ell$,
completing the proof.
\end{proof}


\section{Reciprocal Bernoulli polynomials}
\label{sec:proof-3}

Recall the \textdef{generalized} reciprocal Bernoulli polynomials
\[
  \BC_{n,k}(x) = x^k \, \BN_n(x^{-1}) \quad (n \geq 0, \, k \in \ZZ).
\]
Since $\BN_n(x)$ is a monic polynomial of degree $n$, due to \eqref{eq:bern-poly}
and \eqref{eq:bern-values}, we obtain
\begin{align}
  \BC_{n,k}(0) &= \begin{cases}
    \infty, & \text{if } k < n, \\
    1, & \text{if } k = n, \\
    0, & \text{if } k > n,
  \end{cases} \label{eq:recip-val-0}
\shortintertext{while the reflection formula~\eqref{eq:bern-refl} for $\BN_n(x)$
implies that}
  \BC_{n,k}(1) &= (-1)^n \, \BN_n \quad (n \geq 0, \, k \in \ZZ).
    \label{eq:recip-val-1}
\end{align}

A kind of reflection formula for $\BC_{n,k}(x)$ can be given as follows.

\begin{lemma}
Let $n \geq 1$, $k \in \ZZ$, and $x \neq 0$. Then we have
\[
  \BC_{n,k}(-x) = (-1)^{n+k} \left( \BC_{n,k}(x) + n x^{k-n+1} \right).
\]
\end{lemma}

\begin{proof}
From~\eqref{eq:bern-diff} and~\eqref{eq:bern-refl}, we infer that
\[
  \BN_n(-x) = (-1)^n \, \BN_n(x+1) = (-1)^n ( \BN_n(x) + n x^{n-1} ).
\]
Hence,
\[
  \BC_{n,k}(-x) = (-1)^k x^k \, \BN_n(-x^{-1})
    = (-1)^{n+k} x^k ( \BN_n(x^{-1}) + n x^{-n+1} ),
\]
showing the result.
\end{proof}

Next, we are interested in finding derivatives of $\BC_{n,k}$. For our purpose,
we can restrict the parameters $n$ and $k$ to simplify formulas.
Thanks to Lemma~\ref{lem:deriv-recip} and \eqref{eq:bern-deriv}, we obtain
the following formula involving the Lah numbers.

\begin{lemma}
For $n \geq \ell \geq 1$, we have
\begin{equation} \label{eq:recip-deriv}
  \dop_x^\ell \, \BN_n(x^{-1}) = \sum_{\nu=1}^{\ell} \LN_{\ell,\nu} \,
    x^{-(\ell+\nu)} (n)_\nu \, \BN_{n-\nu}(x^{-1}).
\end{equation}
\end{lemma}

First, we have the following simple rule.

\begin{lemma} \label{lem:recip-deriv}
For $n, k \geq 1$, we have
\[
  \BC'_{n,k}(x) = k \BC_{n,k-1}(x) - n \BC_{n-1,k-2}(x).
\]
\end{lemma}

\begin{proof}
Evaluating the derivative by using \eqref{eq:bern-deriv} yields that
\[
  \BC'_{n,k}(x) = k x^{k-1} \, \BN_n(x^{-1}) - n x^k \, \BN_{n-1}(x^{-1}) / x^2.
    \qedhere
\]
\end{proof}

More generally, we arrive at the following result.

\begin{theorem} \label{thm:recip-deriv}
For $n, k \geq \ell \geq 0$, we have
\begin{equation} \label{eq:recip-deriv-2}
  \dop_x^\ell \, \BC_{n,k}(x) = \ell! \sum_{\nu=0}^{\ell} (-1)^\nu
    \binom{k-\nu}{k-\ell} \binom{n}{\nu} \BC_{n-\nu,k-\ell-\nu}(x).
\end{equation}
In particular, we have
\begin{equation} \label{eq:recip-deriv-3}
  \BC_{n,k}^{(\ell)}(1) = (-1)^n \ell! \sum_{\nu=0}^{\ell}
    \binom{k-\nu}{k-\ell} \binom{n}{\nu} \BN_{n-\nu}.
\end{equation}
\end{theorem}

\begin{proof}
Initially, we consider \eqref{eq:recip-deriv-2}. Since the case $\ell=0$ is
trivial, assume that $\ell \geq 1$. By Leibniz's rule, we can write
\begin{equation} \label{eq:deriv-expr}
  \dop_x^\ell \, \BC_{n,k}(x) = \sum_{j=0}^{\ell} \binom{\ell}{j}
    \big( \BN_n(x^{-1}) \big)^{(j)} \big( x^k \big)^{(\ell-j)}.
\end{equation}
Using~\eqref{eq:recip-deriv} and replacing the Lah numbers by~\eqref{eq:Lah-def},
we obtain
\[
  \dop_x^j \, \BN_n(x^{-1}) = (-1)^j j! \sum_{\nu=1}^{j} \binom{j-1}{\nu-1}
    x^{-(j+\nu)} \binom{n}{\nu} \BN_{n-\nu}(x^{-1}) \quad (1 \leq j \leq \ell).
\]
Allowing the convention $\binom{a}{-1} = \delta_{a,-1}$ by using Kronecker's
delta, we can extend the above summation by starting from index $\nu=0$.
This covers the case $j=0$, where the sum then equals $\BN_n(x^{-1})$.
Further, we have
\[
  \dop_x^{\ell-j} x^k = (k)_{\ell-j} \, x^{k-(\ell-j)}.
\]
Putting all of this together provides
\[
  \dop_x^\ell \, \BC_{n,k}(x) = \ell! \sum_{j=0}^{\ell} (-1)^j \binom{k}{\ell-j}
    \sum_{\nu=0}^{j} \binom{j-1}{\nu-1} \binom{n}{\nu} x^{k-\ell-\nu} \, \BN_{n-\nu}(x^{-1}).
\]
By changing the summation, this further simplifies with some coefficients $c_\nu$ to
\[
  \dop_x^\ell \, \BC_{n,k}(x) = \ell! \sum_{\nu=0}^{\ell} c_\nu \binom{n}{\nu}
    \BC_{n-\nu,k-\ell-\nu}(x).
\]
Now, we compute the coefficients $c_\nu$. By regarding the condition $j \geq \nu$
due to vanishing binomial coefficients, we obtain for $0 \leq \nu \leq \ell$ that
\[
  c_\nu = \sum_{j=\nu}^{\ell} (-1)^j \binom{k}{\ell-j} \binom{j-1}{\nu-1}.
\]
The case $\nu=0$ reduces to
\[
  c_0 = \binom{k}{\ell} = \binom{k}{k-\ell}.
\]
In the other case $\nu \geq 1$, we shift the index $j$ such that
\[
  c_\nu = \sum_{j=0}^{\ell-\nu} (-1)^{j+\nu} \binom{k}{\ell-\nu-j} \binom{j+\nu-1}{\nu-1}
    = (-1)^\nu \sum_{j=0}^{\ell-\nu} (-1)^j \binom{\nu+j-1}{j} \binom{k}{\ell-\nu-j}.
\]
By applying the alternative Chu-Vandermonde identity~\eqref{eq:Chu-Vand-alt},
we achieve that
\[
  c_\nu = (-1)^\nu \binom{k-\nu}{\ell-\nu} = (-1)^\nu \binom{k-\nu}{k-\ell}.
\]
Finally, the coefficients $c_\nu$ in both cases imply~\eqref{eq:recip-deriv-2}.
From relation~\eqref{eq:recip-val-1}, we then further infer~\eqref{eq:recip-deriv-3}.
This proves the theorem.
\end{proof}

Obviously, one could also derive a proof of~\eqref{eq:recip-deriv-2} by
induction using Lemma~\ref{lem:recip-deriv}. Here, we used a direct way to show
the origin of formula~\eqref{eq:recip-deriv-2}, as well as a certain connection
with the Lah numbers via~\eqref{eq:recip-deriv}.

\begin{remark}
Since the Bernoulli polynomials $\BN_n(x)$ form an Appell sequence, the results
in this section can be easily reformulated for any sequence of polynomials
$(p_n)_{n \geq 0}$ satisfying ${p'_n = n p_{n-1}}$ for ${n \geq 1}$.
See the follow-up paper \cite{Kellner:2021} for such results applied to Appell
polynomials.
\end{remark}

Now, we are ready to give a proof of Theorem~\ref{thm:coeff-deriv} below by
specializing some results. Define the functions
\begin{equation} \label{eq:PF-def}
  \PF_{n,k}(x) = \sum_{\nu=0}^{k} \binom{n}{\nu} (2k-\nu)_k \, x^{n-\nu}
    \quad (0 \leq k \leq n).
\end{equation}
Actually, the following recurrence of the functions $\PF_{n,k}(x) \in \ZZ[x]$
is essential for the proof.

\begin{lemma} \label{lem:PF-recurr}
Let $n - 2 \geq k \geq 0$. Then we have the recurrence
\begin{equation} \label{eq:PF-recurr}
  n(n-1) \PF_{n-2,k}(x) = \PF_{n,k+2}(x) - 2(2k+3) \PF_{n,k+1}(x).
\end{equation}
\end{lemma}

\begin{proof}
Let $n - 2 \geq k \geq 0$. First we consider the left-hand side of
\eqref{eq:PF-recurr}. Note that
\[
  n(n-1) \binom{n-2}{\nu-2} (2k-(\nu-2))_k
    = (\nu)_2 \binom{n}{\nu} (2k+2-\nu)_k
\]
for $2 \leq \nu \leq n$, and $(\nu)_2 = 0$ for $\nu = 1,2$.
By shifting the index $\nu$ and extending the summation, we can write
\begin{equation} \label{eq:PF-lhs}
  n(n-1) \PF_{n-2,k}(x)
    = \sum_{\nu=0}^{k+2} \binom{n}{\nu} (\nu)_2 (2k+2-\nu)_k \, x^{n-\nu}.
\end{equation}

We collect the coefficients of \eqref{eq:PF-recurr}
regarding $x^{n-\nu}$ for a fixed $\nu \in \set{0,\ldots,k+2}$.
The left-hand side yields $L = \binom{n}{\nu} (\nu)_2 (2k+2-\nu)_k$
using \eqref{eq:PF-lhs}, while the right-hand side provides
\[
  R = \binom{n}{\nu} (2(k+2)-\nu)_{k+2} - 2(2k+3) \binom{n}{\nu} (2(k+1)-\nu)_{k+1}.
\]
Note that for $\nu = k+2$ the coefficient in $\PF_{n,k+1}(x)$ vanishes,
which coincides in that case with
\[
  (2(k+1)-\nu)_{k+1} = (k)_{k+1} = 0,
\]
since $k \geq 0$. After factoring terms out, we get
$L - R = \binom{n}{\nu} (2k+2-\nu)_k \, T$,
where we infer for the remaining term the identity
\[
  T = (\nu)_2 - (2k+4-\nu)_2 + 2(2k+3)(k+2-\nu) = 0.
\]
This implies the result.
\end{proof}

\begin{remark}
One can extend Lemma~\ref{lem:PF-recurr} as follows. Since recurrence
\eqref{eq:PF-recurr} is linear, one can replace $x^{n-\nu}$ with terms of
any sequence, say $s_{n-\nu}$.
\end{remark}

\begin{proof}[Proof of Theorem~\ref{thm:coeff-deriv}]
For $n \geq k \geq 0$, we find by using Theorem~\ref{thm:recip-deriv} that
\begin{equation} \label{eq:recip-spec}
  \BC_{n,2k}^{(k)}(1) = (-1)^n k! \sum_{\nu=0}^{k}
    \binom{2k-\nu}{k} \binom{n}{\nu} \BN_{n-\nu}.
\end{equation}
Set $\bb_{n,k} = \BC_{n,2k}^{(k)}(1)$. One may observe the similarity
between~\eqref{eq:PF-def} and~\eqref{eq:recip-spec}.
Considering the above remark, we replace $x^{n-\nu}$ by $\BN_{n-\nu}$ in
\eqref{eq:PF-def} and apply Lemma~\ref{lem:PF-recurr}.
Thus, recurrence~\eqref{eq:PF-recurr} also holds for $\bb_{n,k}$ in place
of $\PF_{n,k}(x)$, being compatible with the parity of the sign $(-1)^n$
in~\eqref{eq:recip-spec}. More precisely, for $n \geq 2$ and
$0 \leq k \leq n-2$, we obtain
\begin{equation} \label{eq:recurr-b}
  \bb_{n,k+2} = 2 (2k+3) \, \bb_{n,k+1} + n(n-1) \, \bb_{n-2,k}.
\end{equation}

Now, let $n \geq 5$ be odd. We compare values of $\bb_{n,k}$ with $\fh_{n,k}$
for small $k$. By Theorem~\ref{thm:coeff-recurr} we can compute $\ff_{n,k}$ for
$k=0,1$, see Table~\ref{tbl:F-coeff-odd}. Using substitution~\eqref{eq:coeff-subst}
and formula~\eqref{eq:recip-spec}, we infer that
\begin{equation} \label{eq:init-fb-1}
  \fh_{n,0} = \bb_{n,1} = -n \BN_{n-1}, \quad
    \fh_{n,1} = \bb_{n,2} = -6n \BN_{n-1},
\end{equation}
as well as
\begin{equation} \label{eq:init-fb-2}
  \fh_{3,0} = \bb_{3,1} = -3 \BN_2 = -\tfrac{1}{2}.
\end{equation}

Furthermore, we can rewrite recurrence~\eqref{eq:coeff-recurr-2}
of Theorem~\ref{thm:FP-recurr-2} by~\eqref{eq:coeff-subst} as
\begin{equation} \label{eq:recurr-f}
  \fh_{n,k+2} = 2 (2k+5) \, \fh_{n,k+1} + n(n-1) \, \fh_{n-2,k},
\end{equation}
being valid for $k \geq 0$. Note that~\eqref{eq:recurr-b}, when shifted by
$k \mapsto k+1$, and~\eqref{eq:recurr-f} describe the same recurrence.
Since $\fh_{n,k} = \bb_{n,k+1}$ for the initial values in~\eqref{eq:init-fb-1}
and~\eqref{eq:init-fb-2}, the equality follows in full for odd $n \geq 3$
and $0 \leq k < n$ by induction using recurrences~\eqref{eq:recurr-b}
and~\eqref{eq:recurr-f}.
It remains to show that for odd $n \geq 3$, we have the property
\begin{equation} \label{eq:coeff-pos}
  (-1)^{\dd_n+1} \, \fh_{n,k} > 0 \quad (0 \leq k \leq \dd_n).
\end{equation}
The goal is now to use recurrence~\eqref{eq:recurr-f} reversely.
For $n=3$, property~\eqref{eq:coeff-pos} is valid by~\eqref{eq:init-fb-2}.
Inductive step $n-2$ to $n$ for odd $n \geq 5$:
We can rewrite~\eqref{eq:recurr-f} multiplied by $(-1)^{\dd_n+1}$ as
\begin{equation} \label{eq:recurr-f2}
  (-1)^{\dd_n+1} \, 2 (2k+5) \, \fh_{n,k+1}
    = (-1)^{\dd_n+1} \, \fh_{n,k+2} + (-1)^{\dd_{n-2}+1} \, n(n-1) \, \fh_{n-2,k},
\end{equation}
noting that $\dd_{n-2} = \dd_n - 1$.
We shall see that all corresponding terms in \eqref{eq:recurr-f2} are
nonnegative. By assumption $\fh_{n-2,k}$ satisfies~\eqref{eq:coeff-pos} for
$n-2$, so for $0 \leq k < \dd_n$.
Note that $\fh_{n,\dd_n+1} = 0$. For $k = \dd_n-1, \ldots, 0$, we iteratively
infer by~\eqref{eq:recurr-f2} that $\fh_{n,k}$ satisfies~\eqref{eq:coeff-pos}
for $1 \leq k \leq \dd_n$.
In case $k=0$, we have $(-1)^{\dd_n+1} \, \fh_{n,0} > 0$
by~\eqref{eq:init-fb-1}, since $(-1)^{(n-3)/2} \, \BN_{n-1} > 0$.
Finally, property~\eqref{eq:coeff-pos} is valid for $n$, showing the induction.
This proves the theorem.
\end{proof}


\section{Symmetry properties}
\label{sec:sym}

Recall the numbers (see Table~\ref{tbl:BF-coeff})
\[
  \BF_{n,k} = \sum_{\nu = k}^{n} \binom{n}{\nu} \binom{\nu}{k} \BN_\nu
    \quad (0 \leq k \leq n),
\]
and define the related functions
\[
  \BF_n(x) = \sum_{k=0}^{n} \BF_{n,k} \, x^k \quad (n \geq 0).
\]

We further need the Genocchi numbers $\GN_n$, which may be defined by
\[
  \frac{2t}{e^t+1} = \sum_{n=0}^{\infty} \GN_n \frac{t^n}{n!}
    \quad (\norm{t} < \pi).
\]
These numbers are intimately connected with the Bernoulli numbers by
\begin{equation} \label{eq:G-B}
  \GN_n = 2 ( 1 - 2^n ) \BN_n \quad (n \geq 0).
\end{equation}
In particular, $\GN_n = 0$ for $n=0$ and odd $n \geq 3$,
while the numbers $\GN_n$ are odd integers for $n=1$ and even
$n \geq 2$, see Comtet~\cite[pp.\,48--49]{Comtet:1974}.
The first few nonzero values are
\[
  \GN_1 = 1,\;\; \GN_2 = -1,\;\; \GN_4 = 1,\;\; \GN_6 = -3,\;\; \GN_8 = 17,\;\; \GN_{10} = -155.
\]

The numbers were introduced by Genocchi~\cite{Genocchi:1852} in 1852,
and were later named after him. Already in 1755,
Euler~\cite[Chap.\,VII, pp.\,496--497]{Euler:1755} considered these numbers
and used~\eqref{eq:G-B} to give a table of factorizations of the first few
numbers~$\GN_{2n}$.

A polynomial $p(x) = a_n x^n + \dotsb + a_1 x + a_0$ is called
\textdef{self-reciprocal} or \textdef{palindromic} when $p(x) = x^n p(x^{-1})$,
since the coefficients then satisfy $a_\nu = a_{n-\nu}$ for $0 \leq \nu \leq n$.
Alternatively, if the condition $a_\nu = -a_{n-\nu}$ (respectively,
$\norm{a_\nu} = \norm{a_{n-\nu}}$) holds, then $p$ is called
\textdef{antipalindromic} (respectively, \textdef{quasipalindromic}).
Note that the definitions do not depend on the degree of~$p$, since it can
happen that $a_n = a_0 = 0$, so $n \neq \deg p$ in that case.

\begin{theorem} \label{thm:BF-reflect}
Let $n \geq 0$. For $0 \leq k \leq n$, we have the reflection relation
\[
  \BF_{n,k} = (-1)^n \, \BF_{n,n-k}.
\]
Equivalently, the polynomials $\BF_n(x)$ are quasipalindromic such that
\[
  \BF_n(x) = (-1)^n x^n \, \BF_n(x^{-1}).
\]
In particular, we have the relations
\[
  \BF_{n,0} = (-1)^n \, \BN_n, \quad \BF_{n,n} = \BN_n, \andq
    \BF_n(x) = \BC_{n,n}(x+1).
\]
Moreover, special values are given by
\[
  \BF_n(-1) = 1, \quad \BF_n(0) = (-1)^n \, \BN_n, \andq
    \BF_n(1) = \BN_n + \tfrac{1}{2} \GN_n.
\]
\end{theorem}

\begin{proof}
If $k = 0$, then we obtain by~\eqref{eq:bern-poly}
and~\eqref{eq:bern-refl} that
\[
  \BF_{n,0} = \BN_n(1) = (-1)^n \, \BN_n = (-1)^n \, \BF_{n,n}.
\]
For $1 \leq k \leq n$, we evaluate the derivative $\BC_{n,n}^{(k)}(1)$
in two different ways. On the one hand, we have by definition that
\[
  \BC_{n,n}(x) = x^n \, \BN_n(x^{-1})
    = \sum_{\nu = 0}^{n} \binom{n}{\nu} \BN_{\nu} \, x^{\nu}.
\]
Thus, we first infer that
\[
  \BC_{n,n}^{(k)}(1) = k! \sum_{\nu = k}^{n} \binom{n}{\nu}
    \binom{\nu}{k} \BN_{\nu} = k! \, \BF_{n,k}.
\]
On the other hand, Theorem~\ref{thm:recip-deriv} implies that
\[
  \BC_{n,n}^{(k)}(1) = (-1)^n k! \sum_{\nu=0}^{k}
    \binom{n}{\nu} \binom{n-\nu}{n-k} \BN_{n-\nu}
    = (-1)^n k! \, \BF_{n,n-k}.
\]
This shows the reflection relation $\BF_{n,k} = (-1)^n \, \BF_{n,n-k}$,
which also implies that $\BF_n(x)$ is quasipalindromic.
Since $\BC_{n,n}^{(k)}(1) = k! \, \BF_{n,k} = \BF_n^{(k)}(0)$ for
$0 \leq k \leq n$, it follows that $\BC_{n,n}(x+1) = \BF_n(x)$.
It remains to compute the special values of $\BF_n(x)$.
By~\eqref{eq:recip-val-0}, we have $\BF_n(-1) = \BC_{n,n}(0) = 1$.
Further, $\BF_n(0) = \BF_{n,0} = (-1)^n \, \BN_n$. Finally,
using~\eqref{eq:bern-val-half} and~\eqref{eq:G-B}, we infer that
\[
  \BF_n(1) = \BC_{n,n}(2) = 2^n \, \BN_n \mleft( \tfrac{1}{2} \mright)
    = 2^n ( 2^{1-n} - 1 ) \BN_n = \BN_n + \tfrac{1}{2} \GN_n. \qedhere
\]
\end{proof}

\begin{corollary}
For $n \geq 1$, we have
$\BF_{n,0} + \dotsb + \BF_{n,n-1} = \frac{1}{2} \GN_n$.
\end{corollary}

Finally, to generalize the results we define the formal power series
\begin{equation} \label{eq:BF-poly}
  \BF_{n,k}(x) = \BF_n(x) \, (x+1)^{k-n}
    = \BF_n(x) \sum_{\nu \geq 0} \binom{k-n}{\nu} x^\nu
    \quad (n \geq 0, \, k \in \ZZ).
\end{equation}

\begin{theorem} \label{thm:BF-deriv}
Let $n \geq 0$ and $k \in \ZZ$. Then we have
\begin{equation} \label{eq:recip-BF}
  \BC_{n,k}(x+1) = \BF_{n,k}(x),
\end{equation}
where we require $\norm{x} < 1$ in the case $k < n$.
In particular, we have for $n, k \geq \ell \geq 0$ that
\begin{equation} \label{eq:recip-BF-2}
  \BC_{n,k}^{(\ell)}(1)
    = \ell! \sum_{\nu=0}^{\ell} \binom{k-n}{\ell-\nu} \BF_{n,\nu}.
\end{equation}
\end{theorem}

\begin{proof}
Since $\BC_{n,n}(x+1) = \BF_n(x)$ by Theorem~\ref{thm:BF-reflect},
we obtain \eqref{eq:recip-BF} by its definition and using the substitution
in \eqref{eq:BF-poly}. We have to consider the expansion
$(x+1)^\alpha = \sum_{\nu \geq 0} \binom{\alpha}{\nu} x^\nu$ with $\alpha = k-n$.
If $\alpha < 0$ so $k < n$, then we have to require $\norm{x} < 1$ for
convergence. To show~\eqref{eq:recip-BF-2}, we apply Leibniz's rule to get
\[
  \dop_x^\ell \, \BC_{n,k}(x+1) = \sum_{\nu=0}^{\ell} \binom{\ell}{\nu}
    \big( \BF_n(x) \big)^{(\nu)} \big( (x+1)^{\alpha} \big)^{(\ell-\nu)}.
\]
Evaluating this equation at $x=0$ easily yields the result.
\end{proof}

As an application, we derive the following special case.

\begin{proof}[Proof of Theorem~\ref{thm:coeff-deriv-2}]
The computation of $\BC_{n,2k}^{(k)}(1)$ follows from Theorem~\ref{thm:BF-deriv}
and \eqref{eq:recip-BF-2} with $\ell = k$ and $k$ replaced by $2k$.
\end{proof}

Let $\coeff{t^n}$ be the linear functional that gives the coefficient of $t^n$
of a formal power series such that
\[
  f(t) = \sum_{n \geq 0} a_n t^n, \quad \coeff{t^n} \, f(t) = a_n.
\]

\begin{theorem} \label{thm:BF-poly}
If $k \geq n \geq 0$, then $\BF_{n,k}(x)$ is a polynomial.
Moreover, it is quasipalin\-dromic such that
\begin{equation} \label{eq:BF-poly-recip}
  \BF_{n,k}(x) = (-1)^n x^k \, \BF_{n,k}(x^{-1}).
\end{equation}
For odd $n \geq 3$ and $k \geq 0$, we have
\begin{equation} \label{eq:FP-BF-coeff}
  \ff_{n,k} = (-1)^{k+1} \frac{2^{k+2}}{k+2} \, \coeff{x^{k+1}} \, \BF_{n,2(k+1)}(x).
\end{equation}
In particular, the central coefficient of $\BF_{n,2k}(x)$ vanishes for
$k \geq (n+1)/2$ as well as the coefficient~$\ff_{n,k}$ of the Faulhaber
polynomial $F_n$ for $k \geq \dd_n + 1$.
\end{theorem}

\begin{proof}
First, let $k \geq n \geq 0$. Then we have by \eqref{eq:BF-poly} that
$\BF_{n,k}(x) = \BF_n(x) (x+1)^\alpha$, with $\alpha = k-n \geq 0$, is a polynomial.
The expansion $(x+1)^\alpha = \sum_{\nu \geq 0} \binom{\alpha}{\nu} x^\nu$ is
palindromic due to the symmetry of the binomial coefficients.
It is well known that the product
\[
  \BF_n(x) (x+1)^\alpha = (-1)^n x^n \, \BF_n(x^{-1}) \, x^\alpha (x^{-1}+1)^\alpha
\]
is also quasipalindromic. Together with Theorem~\ref{thm:BF-reflect} this shows
\eqref{eq:BF-poly-recip}.

Now let $n \geq 3$ be odd. As a result, $\BF_{n,2k}(x)$ is an antipalindromic
polynomial for $k \geq (n+1)/2$. In this case, we have
$\coeff{x^k} \, \BF_{n,2k}(x) = - \coeff{x^k} \, \BF_{n,2k}(x)$ by \eqref{eq:BF-poly-recip},
implying that the central coefficient of $\BF_{n,2k}(x)$ vanishes.
For $k \geq 0$, we have by definition that
\[
  \coeff{x^k} \, \BF_{n,2k}(x) = \frac{1}{k!} \, \BC_{n,2k}^{(k)}(1).
\]
From Corollary~\ref{cor:coeff-deriv}, we then deduce \eqref{eq:FP-BF-coeff}
for ${0 \leq k < n}$. Furthermore, $\ff_{n,k} = 0$ for $k \geq \dd_n + 1 = (n-1)/2$
by definition, which coincides with $\coeff{x^{k+1}} \, \BF_{n,2(k+1)}(x) = 0$.
Hence, \eqref{eq:FP-BF-coeff} holds for all $k \geq 0$, completing the proof.
\end{proof}

See Example~\ref{exp:BF-poly} for the computation of the coefficients
$\ff_{n,k}$ for ${n=7}$ by means of Theorem~\ref{thm:BF-poly}. We also derive the
following interpretation: Since the central coefficient of $\BF_{n,2k}(x)$
vanishes by symmetry for odd $n \geq 1$ and $k \geq (n+1)/2$, we then have
\[
  \coeff{x^k} \, \BF_{n,2k}(x) = \frac{1}{k!} \, \BC_{n,2k}^{(k)}(1) = 0.
\]
Consequently, the recurrences of Corollaries~\ref{cor:bern-recurr} and~\ref{cor:bern-recurr-2}
are induced by this symmetry property, giving another point of view.

Regarding Theorem~\ref{thm:BF-poly} and formula~\eqref{eq:FP-BF-coeff},
we can go further to show that the coefficients of the Faulhaber polynomials
have an explanation in a \textsl{natural} context using mainly the coefficients
$\coeff{x^k} \, \BF_{n,2k}(x)$, which we interpret as
\textdef{central coefficients} for all $k \geq 0$.

Recall Lemma~\ref{lem:bern-refl}, which implies that for odd $n \geq 3$ the
polynomials $\BN_n(x)/(x - \frac{1}{2})$ are symmetric around $x = \frac{1}{2}$.
Further recall Theorem~\ref{thm:Faulhaber-Jacobi} and Lemma~\ref{lem:S12-ident}
for a reformulation of Faulhaber's equations. Both cases for odd and even
indices can be described in one combined formula such that only the above
polynomials $\BN_n(x)/(x - \frac{1}{2})$ with odd indices $n \geq 3$ appear.

So far, the factor $x - \frac{1}{2}$ has not been interpreted in a suitable form.
It is a striking coincidence that we have a substitution with a quadratic term,
e.g., ${y = \binom{x}{2}}$, whose derivative $y' = x - \frac{1}{2}$ matches the
needed factor. In fact, this connection has its simple explanation and
application by derivation and integration, see \eqref{eq:BN-FP-sum}
and \eqref{eq:SN-FP-sum} below. The following theorem and its corollary
reassemble the results.

\begin{theorem} \label{thm:BN-FP-poly}
Let $n \geq 3$ be odd. Set $y = \binom{x}{2}$ and $u = x(x-1)$.
Regard derivatives with respect to $x$. Then Faulhaber's equations
\[
  S_n(x) = y^2 F_n(y) \andq
    S_{n-1}(x) = \tfrac{1}{3} (y^2)' \, F_{n-1}(y)
\]
are each equivalent to the relations
\begin{equation} \label{eq:BN-FP}
  \frac{1}{2} \frac{\BN_n(x)}{x - \frac{1}{2}} = \FP_n(u) \andq
    \BN_n(x) = u' \, \FP_n(u),
\end{equation}
where
\[
  \FP_n(u) = \sum_{k \geq 0} (-1)^k \fp_{n,k} \, u^k \andq
    \fp_{n,k} = \coeff{x^k} \, \BF_{n,2k}(x)
    = \frac{1}{k!} \, \BC_{n,2k}^{(k)}(1) \quad (k \geq 0).
\]
Moreover, we have for $0 \leq k \leq n$ the expressions
\[
  \fp_{n,k} = \sum_{\nu=0}^{k} \binom{2k-n}{k-\nu} \BF_{n,\nu}
    = - \sum_{\nu=0}^{k} \binom{2k-\nu}{k} \binom{n}{\nu} \BN_{n-\nu}.
\]
In particular, $\fp_{n,k} = 0$ for $k=0$ and $k \geq (n+1)/2$;
otherwise, $(-1)^{(n-1)/2} \fp_{n,k} > 0$, so $\deg \FP_n = (n-1)/2$
and $\FP_n(0) = 0$.
\end{theorem}

\begin{proof}
Let $n \geq 3$ be odd. Note that $y = \frac{1}{2} u$,
$y' = x - \frac{1}{2}$, and $u' = 2x - 1$.
Thus, both relations in \eqref{eq:BN-FP} are equivalent.
We consider the first equation $S_n(x) = y^2 F_n(y)$.
By definition we have
\[
  S_n(x) = \frac{1}{n+1}( \BN_{n+1}(x) - \BN_{n+1} )
    = \sum_{k \geq 0} \ff_{n,k} \, y^{k+2}.
\]
Differentiating with respect to $x$ and using \eqref{eq:FP-BF-coeff},
this yields the equations
\begin{equation} \label{eq:BN-FP-sum}
  \frac{\BN_n(x)}{x - \frac{1}{2}}
    = \sum_{k \geq 0} (k+2) \ff_{n,k} \, y^{k+1}
    = 2 \sum_{k \geq 1} (-1)^k \fp_{n,k} \, u^k,
\end{equation}
where $\fp_{n,k} = \coeff{x^k} \, \BF_{n,2k}(x)$ for $k \geq 0$.
Since $n \geq 3$ is odd, we have $\BN_n = 0$, and therefore,
$\BF_n(0) = \BF_{n,0} = 0$ by Theorem~\ref{thm:BF-reflect}.
It then follows that $\BF_{n,0}(0) = 0$, so $\fp_{n,0} = 0$.
Including the case $k=0$ implies \eqref{eq:BN-FP}.
The second Faulhaber equation provides
\[
  S_{n-1}(x) = \frac{1}{n}( \BN_n(x) - \BN_n )
    = \frac{2}{3} \mleft( x - \frac{1}{2} \mright)
    \sum_{k \geq 0} \ff_{n-1,k} \, y^{k+1}.
\]
By using Corollary~\ref{cor:FP-odd-even}, the above equation then transforms
into \eqref{eq:BN-FP-sum}, and we are done. Conversely, starting from
\eqref{eq:BN-FP}, we arrive at the equations with $S_{n-1}(x)$ and $S_n(x)$.
For the latter, note that
\begin{equation} \label{eq:SN-FP-sum}
  S_n(x) = \int_{0}^{x} \BN_n(t) dt
    = \int_{0}^{u} \FP_n(s) ds
    = \sum_{k \geq 0} (-1)^k \fp_{n,k} \frac{u^{k+1}}{k+1}.
\end{equation}
The remaining parts easily follow from Theorems~\ref{thm:coeff-deriv},
\ref{thm:coeff-deriv-2}, and~\ref{thm:BF-poly}.
\end{proof}

See Table~\ref{tbl:FP-poly-odd} for the first few related Faulhaber polynomials
$\FP_n(u)$. Example~\ref{exp:BF-poly} also contains the computation of
$\FP_n(u)$ for $n=7$, compared to $F_n(y)$, by Theorem~\ref{thm:BN-FP-poly}.
For odd indices $n \geq 3$, the relations of \eqref{eq:BN-FP} can be explained
by reciprocal Bernoulli polynomials. Since \eqref{eq:BN-FP} is
also equivalent to $n \, S_{n-1}(x) / (x - \frac{1}{2}) = 2 \FP_n(u)$,
the polynomial parts in terms of $\xi$ and $u$ of Tables~\ref{tbl:S-poly-even}
and~\ref{tbl:FP-poly-odd}, respectively, are equivalent, only differing by a
factor of~$2$ and alternating signs.
By Theorems~\ref{thm:BF-poly} and~\ref{thm:BN-FP-poly}, the properties of the
coefficients $\fp_{n,k}$ and their relationship with $\BF_{n,2k}(x)$ and
$\BC_{n,2k}^{(k)}(1)$ can be characterized in the table below.

\setcounter{table}{\value{theorem}}

\begin{table}[H] \small
\setstretch{1.5}
\setlength{\tabcolsep}{6pt}
\begin{center}
\begin{tabular}{cccc}
  \toprule
  \multicolumn{1}{c}{Index $k$} & \multicolumn{1}{c}{Coefficients}
    & \multicolumn{1}{c}{$\BF_{n,2k}(x) \in$}
    & \multicolumn{1}{c}{Properties} \\\hline
  $0$ & $\fp_{n,k} = 0$
    & $\QQ[[x]]$ & $\BF_{n,2k}(0) = 0$ \\
  $1, \ldots, \frac{n-1}{2}$ & $(-1)^{\frac{n-1}{2}} \fp_{n,k} > 0$
    & $\QQ[[x]]$ & $(-1)^{\frac{n-1}{2}} \, \coeff{x^k} \, \BF_{n,2k}(x) > 0$ \\
  $\frac{n+1}{2}, \ldots, \infty$ & $\fp_{n,k} = 0$
    & $\QQ[x]$ & $\BF_{n,2k}(x) \text{ is antipalindromic}$ \\
  \bottomrule
\end{tabular}

\caption{Coefficients $\fp_{n,k}$ for odd $n \geq 3$.}
\end{center}
\end{table}

\setcounter{theorem}{\value{table}}

A further variant of the Faulhaber equations can be given as follows.

\begin{corollary}
Let $n \geq 3$ be odd. Set $y = \binom{x}{2}$ and regard derivatives with
respect to~$x$. Then Faulhaber's equations are equivalent to
\[
  S_n(x) = \FF_n(y) \andq S_{n-1}(x) = y' \, \FF_{n-1}(y)
\]
with the polynomials $\FF_n(y) = y^2 F_n(y)$ and
$\FF_{n-1}(y) = \frac{2}{3} y F_{n-1}(y)$, which obey the Appell property
$\FF'_n(y) = n \FF_{n-1}(y)$ regarding the derivative with respect to~$y$.
\end{corollary}

This is a modern reformulation of the equations \eqref{eq:S-odd-even}
from the beginning that were known to Faulhaber, closing the circle.

At the end, we note that Theorem~\ref{thm:BF-reflect} also implies the
following corollary. Since
$\binom{n}{\nu} \binom{\nu}{k} = \binom{n}{k} \binom{n-k}{\nu-k}$,
one can write
\[
  \BF_{n,k} = \binom{n}{k} \sum_{\nu = k}^{n} \binom{n-k}{\nu-k} \BN_\nu
    = \binom{n}{k} \sum_{\nu = 0}^{n-k} \binom{n-k}{\nu} \BN_{k+\nu}.
\]
By changing the indexing and variables with $r = k$ and $s = n-k$,
the reflection relation of~$\BF_{n,k}$ leads to a well-known result
(see Lucas~\cite[Sec.\,135, p.\,240]{Lucas:1891} (1891),
Tits~\cite[p.\,191]{Tits:1922} (1922), and
Gessel~\cite[Lem.\,7.2, p.\,416]{Gessel:2003} (2003) for short proofs
using umbral calculus).

\begin{corollary}
For $r, s \geq 0$, we have the reciprocity relation
\[
  (-1)^r \sum_{\nu=0}^{r} \binom{r}{\nu} \BN_{s+\nu}
    = (-1)^s \sum_{\nu=0}^{s} \binom{s}{\nu} \BN_{r+\nu}.
\]
\end{corollary}

Actually, this reciprocity relation goes back to von Ettingshausen
\cite[Lec.\,42, Eq.\,(65), p.\,284]{Ettingshausen:1827} in 1827,
who used finite differences to achieve the result.


\newpage
\section{Conclusion}
\label{sec:con}

Using the notation of the previous section, the main results can be summarized
as follows. We consider the distinguished case, where $n \geq 3$ is odd.
The relationship between the power sum $S_n(x)$ and the Faulhaber polynomial
$\FF_n(y)$ can be given by a short chain of relations
\[
  S_n(x) = \int_{0}^{x} \BN_n(t) dt
    = \int_{0}^{x} u' \, \FP_n(u) dt
    = \int_{0}^{2y} \FP_n(u) du
    = \FF_n(y),
\]
where $u = t(t-1)$ and $y = \binom{x}{2}$ are quadratic substitutions.
The related Faulhaber polynomial $\FP_n(u)$ is defined by
\[
  \FP_n(u) = \sum_{k \geq 0} (-1)^k \fp_{n,k} \, u^k.
\]

Our new approach allows us to establish the unexpected connection between
the coefficients $\fp_{n,k}$ and the introduced generalized (self-) reciprocal
Bernoulli polynomials such that
\[
  \fp_{n,k} = \coeff{x^k} \, \BF_{n,2k}(x)
    = \frac{1}{k!} \, \BC_{n,2k}^{(k)}(1) \quad (k \geq 0).
\]
These relations also enable us to compute the coefficients $\fp_{n,k}$
in several ways (see Examples~\ref{exp:BF-poly} and~\ref{exp:coeff} for the
related coefficients $\ff_{n,k}$).
For $1 \leq k \leq (n-1)/2$, the nonvanishing of the coefficients $\fp_{n,k}$,
having the same sign for fixed $n$, follows from the property
\[
  (-1)^{(n-1)/2} \, \coeff{x^k} \, \BF_{n,2k}(x) > 0.
\]

On the other side, the coefficients $\fp_{n,k}$ vanish for $k \geq (n+1)/2$,
since $\deg \FP_n = (n-1)/2$. Additionally, we also have for $k \geq (n+1)/2$
the chain of implications induced by symmetry properties:
\[
  \BF_{n,2k}(x) \text{ is antipalindromic}
    \;\implies\; \coeff{x^k} \, \BF_{n,2k}(x) = 0
    \;\implies\; \BC_{n,2k}^{(k)}(1) = 0
    \;\implies\; \fp_{n,k} = 0.
\]
For $(n+1)/2 \leq k \leq n$, the expressions for $\BC_{n,2k}^{(k)}(1)$
imply the recurrences
\[
  \sum_{\nu=0}^{k} \binom{2k-n}{k-\nu} \BF_{n,\nu}
    = \sum_{\nu=0}^{k} \binom{2k-\nu}{k} \binom{n}{\nu} \BN_{n-\nu} = 0,
\]
which are causally induced by the antipalindromic property of $\BF_{n,2k}(x)$.

For further reading see the sequel paper~\cite{Kellner:2021}.
Generalizing the results for any Appell sequence $(\AN_\nu(x))_{\nu \geq 0}$,
it turns out that the Appell polynomials $\AN_\nu(x)$ can be similarly described
by certain Faulhaber-type polynomials when the reflection relation
\[
  \AN_\nu(1-x) = (-1)^\nu \, \AN_\nu(x) \quad (\nu \geq 0)
\]
holds as \eqref{eq:bern-refl} for the Bernoulli polynomials.


\newpage
\appendix
\section{Computations}

\setlength{\intextsep}{6pt plus 1pt minus 1pt}

\begin{table}[H] \small
\setstretch{1.25}
\begin{center}
\begin{tabular}{r@{\;=\;}l}
  \toprule
  $F_2(y)$ & $1$ \\
  $F_3(y)$ & $1$ \\
  $F_4(y)$ & $\frac{1}{5} ( 6 y - 1 )$ \\
  $F_5(y)$ & $\frac{1}{3} ( 4 y - 1 )$ \\
  $F_6(y)$ & $\frac{1}{7} ( 12 y^2 - 6 y + 1 )$ \\
  $F_7(y)$ & $\frac{1}{3} ( 6 y^2 - 4 y + 1 )$ \\
  $F_8(y)$ & $\frac{1}{15} ( 40 y^3 - 40 y^2 + 18 y - 3 )$ \\
  $F_9(y)$ & $\frac{1}{5} ( 16 y^3 - 20 y^2 + 12 y - 3 )$ \\
  $F_{10}(y)$ & $\frac{1}{11} ( 48 y^4 - 80 y^3 + 68 y^2 - 30 y + 5 )$ \\
  $F_{11}(y)$ & $\frac{1}{3} ( 16 y^4 - 32 y^3 + 34 y^2 - 20 y + 5 )$ \\
  $F_{12}(y)$ & $\frac{1}{455} ( 3360 y^5 - 8400 y^4 + 11480 y^3 - 9440 y^2 + 4146 y - 691 )$ \\
  $F_{13}(y)$ & $\frac{1}{105} ( 960 y^5 - 2800 y^4 + 4592 y^3 - 4720 y^2 + 2764 y - 691 )$ \\
  \bottomrule
\end{tabular}

\caption{First few Faulhaber polynomials $F_n$ for $n \geq 2$.}
\label{tbl:F-poly}
\end{center}
\end{table}
\vfill

\begin{table}[H] \small
\setstretch{1.25}
\begin{center}
\begin{tabular}{r@{\;=\;}l}
  \toprule
  $S_3(x)$ & $\frac{1}{4} u^2$ \\
  $S_5(x)$ & $\frac{1}{6} u^2 ( u - \frac{1}{2} )$ \\
  $S_7(x)$ & $\frac{1}{8} u^2 ( u^2 - \frac{4}{3} u + \frac{2}{3} )$ \\
  $S_9(x)$ & $\frac{1}{10} u^2 ( u^3 - \frac{5}{2} u^2 + 3 u - \frac{3}{2} )$ \\
  $S_{11}(x)$ & $\frac{1}{12} u^2 ( u^4 - 4 u^3 + \frac{17}{2} u^2 - 10 u + 5 )$ \\
  $S_{13}(x)$ & $\frac{1}{14} u^2 ( u^5 - \frac{35}{6} u^4 + \frac{287}{15} u^3
    - \frac{118}{3} u^2 + \frac{691}{15} u - \frac{691}{30} )$ \\
  \bottomrule
\end{tabular}

\caption{\parbox[t]{19.5em}{First few polynomials $S_n(x)$ for odd $n \geq 3$ in terms
of $u=x(x-1)$ according to Jacobi~\cite{Jacobi:1834}.}}
\label{tbl:S-poly-odd}
\end{center}
\end{table}
\vfill

\begin{table}[H] \small
\setstretch{1.25}
\begin{center}
\begin{tabular}{r@{\;=\;}l}
  \toprule
  $S_2(x)$ & $-\frac{1}{3} (x-\frac{1}{2}) \xi$ \\
  $S_4(x)$ & $+\frac{1}{5} (x-\frac{1}{2}) \xi ( \xi + \frac{1}{3} )$ \\
  $S_6(x)$ & $-\frac{1}{7} (x-\frac{1}{2}) \xi ( \xi^2 + \xi + \frac{1}{3} )$ \\
  $S_8(x)$ & $+\frac{1}{9} (x-\frac{1}{2}) \xi ( \xi^3 + 2 \xi^2 + \frac{9}{5} \xi + \frac{3}{5} )$ \\
  $S_{10}(x)$ & $-\frac{1}{11} (x-\frac{1}{2}) \xi ( \xi^4 + \frac{10}{3} \xi^3
    + \frac{17}{3} \xi^2 + 5 \xi + \frac{5}{3} )$ \\
  $S_{12}(x)$ & $+\frac{1}{13} (x-\frac{1}{2}) \xi ( \xi^5 + 5 \xi^4
    + \frac{41}{3} \xi^3 + \frac{472}{21} \xi^2 + \frac{691}{35} \xi + \frac{691}{105} )$ \\
  \bottomrule
\end{tabular}

\caption{\parbox[t]{21em}{First few polynomials $S_n(x)$ for even $n \geq 2$ in terms
of $\xi=-x(x-1)$ according to Schr\"oder~\cite{Schroeder:1867}.}}
\label{tbl:S-poly-even}
\end{center}
\end{table}

\begin{table}[H] \small
\setstretch{1.25}
\begin{center}
\begin{tabular}{r@{\;=\;}lcr@{\;=\;}l}
  \toprule
  $\ff_{n,0}$ &$2n \BN_{n-1}$ &&
  $\ff_{n,2}$ &$4 \left( 10 n \BN_{n-1} + \binom{n}{3} \BN_{n-3} \right)$ \\
  $\ff_{n,1}$ &$-8n \BN_{n-1}$ &&
  $\ff_{n,3}$ &$-32 \left( 7 n \BN_{n-1} + \binom{n}{3} \BN_{n-3} \right)$ \\
  \bottomrule
\end{tabular}

\caption{\parbox[t]{22em}{First few coefficients $\ff_{n,k}$ of Faulhaber
polynomials $F_n$ for odd $n \geq 3$ and $0 \leq k \leq \dd_n = (n-3)/2$.}}
\label{tbl:F-coeff-odd}
\end{center}
\end{table}
\vfill

\begin{table}[H] \small
\setstretch{1.25}
\begin{center}
\begin{tabular}{r@{\;=\;}lcr@{\;=\;}l}
  \toprule
  $\ff_{n,0}$ &$6 \BN_{n}$ &&
  $\ff_{n,2}$ &$8 \left( 30 \BN_{n} + \binom{n-1}{2} \BN_{n-2} \right)$ \\
  $\ff_{n,1}$ &$-36 \BN_{n}$ &&
  $\ff_{n,3}$ &$-80 \left( 21 \BN_{n} + \binom{n-1}{2} \BN_{n-2} \right)$ \\
  \bottomrule
\end{tabular}

\caption{\parbox[t]{22em}{First few coefficients $\ff_{n,k}$ of Faulhaber
polynomials $F_n$ for even $n \geq 2$ and $0 \leq k \leq \dd_n = (n-2)/2$.}}
\label{tbl:F-coeff-even}
\end{center}
\end{table}
\vfill

\begin{table}[H] \small
\newcommand{\mc}[1]{\multicolumn{1}{r}{\hspace*{1.5em}$#1$}}
\setstretch{1.25}
\begin{center}
\begin{tabular}{c|*{8}{r}}
  \toprule
  \diagbox[width=2.2em, height=2.2em]{$n$}{$k$} &
    \mc{0} & \mc{1} & \mc{2} & \mc{3} & \mc{4} & \mc{5} & \mc{6} & \mc{7} \\\hline
  $0$ & $1$ \\
  $1$ & $\frac{1}{2}$ & $0$ \\
  $2$ & $\frac{1}{6}$ & $-\frac{2}{3}$ & $-2$ \\
  $3$ & $0$ & $-\frac{1}{2}$ & $0$ & $0$ \\
  $4$ & $-\frac{1}{30}$ & $-\frac{1}{15}$ & $\frac{8}{5}$ & $8$ & $88$ \\
  $5$ & $0$ & $\frac{1}{6}$ & $1$ & $0$ & $0$ & $0$ \\
  $6$ & $\frac{1}{42}$ & $\frac{1}{21}$ & $-\frac{5}{7}$ & $-\frac{64}{7}$ & $-80$ & $-1200$ & $-23760$ \\
  $7$ & $0$ & $-\frac{1}{6}$ & $-1$ & $-3$ & $0$ & $0$ & $0$ & $0$ \\
  \bottomrule
\end{tabular}

\caption{First few coefficients $\bb_{n,k}$.}
\label{tbl:bb-coeff}
\end{center}
\end{table}
\vfill

\begin{table}[H] \small
\newcommand{\mc}[1]{\multicolumn{1}{r}{\hspace*{1.5em}$#1$}}
\setstretch{1.25}
\begin{center}
\begin{tabular}{c|*{9}{r}}
  \toprule
  \diagbox[width=2.2em, height=2.2em]{$n$}{$k$} &
    \mc{0} & \mc{1} & \mc{2} & \mc{3} & \mc{4} & \mc{5} & \mc{6} & \mc{7} & \mc{8} \\\hline
  $0$ & $1$ \\
  $1$ & $\frac{1}{2}$ & $-\frac{1}{2}$ \\
  $2$ & $\frac{1}{6}$ & $-\frac{2}{3}$ & $\frac{1}{6}$ \\
  $3$ & $0$ & $-\frac{1}{2}$ & $\frac{1}{2}$ & $0$ \\
  $4$ & $-\frac{1}{30}$ & $-\frac{2}{15}$ & $\frac{4}{5}$ & $-\frac{2}{15}$ & $-\frac{1}{30}$ \\
  $5$ & $0$ & $\frac{1}{6}$ & $\frac{2}{3}$ & $-\frac{2}{3}$ & $-\frac{1}{6}$ & $0$ \\
  $6$ & $\frac{1}{42}$ & $\frac{1}{7}$ & $-\frac{1}{7}$ & $-\frac{32}{21}$ & $-\frac{1}{7}$
   & $\frac{1}{7}$ & $\frac{1}{42}$ \\
  $7$ & $0$ & $-\frac{1}{6}$ & $-1$ & $-\frac{4}{3}$ & $\frac{4}{3}$ & $1$ & $\frac{1}{6}$ & $0$ \\
  $8$ & $-\frac{1}{30}$ & $-\frac{4}{15}$ & $-\frac{4}{15}$ & $\frac{32}{15}$ & $\frac{16}{3}$
   & $\frac{32}{15}$ & $-\frac{4}{15}$ & $-\frac{4}{15}$ & $-\frac{1}{30}$ \\
  \bottomrule
\end{tabular}

\caption{First few coefficients $\BF_{n,k}$.}
\label{tbl:BF-coeff}
\end{center}
\end{table}

\begin{table}[H] \small
\setstretch{1.25}
\begin{center}
\begin{tabular}{r@{\;=\;}l}
  \toprule
  $\FP_3(u)$ & $\frac{1}{2} u$ \\
  $\FP_5(u)$ & $\frac{1}{2} u^2 - \frac{1}{6} u$ \\
  $\FP_7(u)$ & $\frac{1}{2} u^3 - \frac{1}{2} u^2 + \frac{1}{6} u$ \\
  $\FP_9(u)$ & $\frac{1}{2} u^4 - u^3 + \frac{9}{10} u^2 - \frac{3}{10} u$ \\
  $\FP_{11}(u)$ & $\frac{1}{2} u^5 - \frac{5}{3} u^4 + \frac{17}{6} u^3
    - \frac{5}{2} u^2 + \frac{5}{6} u$ \\
  $\FP_{13}(u)$ & $\frac{1}{2} u^6 - \frac{5}{2} u^5 + \frac{41}{6} u^4
    - \frac{236}{21} u^3 + \frac{691}{70} u^2 - \frac{691}{210} u$ \\
  \bottomrule
\end{tabular}

\caption{First few polynomials $\FP_n$ for odd $n \geq 3$.}
\label{tbl:FP-poly-odd}
\end{center}
\end{table}
\vfill

\setcounter{theorem}{\value{table}}

\begin{example} \label{exp:BF-poly}
Computation of the coefficients $\ff_{n,k}$ for $n=7$ using the power series
of $\BF_{n,2k}(x)$. By Theorem~\ref{thm:BF-poly} we obtain
\newcommand{\mcl}[1]{{\color{purple}\mathbf{#1}}}
\begin{alignat*}{2}
  \BF_{7,0}(x) &= -\tfrac{1}{6} x + \tfrac{1}{6} x^2 + x^3
    - \tfrac{10}{3} x^4 + \tfrac{10}{3} x^5 + \dotsm\!, \\
  \BF_{7,2}(x) &= \mcl{-\tfrac{1}{6}} x - \tfrac{1}{6} x^2 + \tfrac{7}{6} x^3
    - \tfrac{7}{6} x^4 - \tfrac{7}{3} x^5 + \dotsm\!, &\quad
  \ff_{7,0} &= -\tfrac{2^2}{2} \times \left( -\tfrac{1}{6} \right) = \tfrac{1}{3}, \\
  \BF_{7,4}(x) &= -\tfrac{1}{6} x \mcl{-\tfrac{1}{2}} x^2 + \tfrac{2}{3} x^3
    + x^4 - \tfrac{7}{2} x^5 + \dotsm\!, &\quad
  \ff_{7,1} &= +\tfrac{2^3}{3} \times \left( -\tfrac{1}{2} \right) = -\tfrac{4}{3}, \\
  \BF_{7,6}(x) &= -\tfrac{1}{6} x - \tfrac{5}{6} x^2 \mcl{-\tfrac{1}{2}} x^3
    + \tfrac{11}{6} x^4 - \tfrac{5}{6} x^5 + \dotsm\!, &\quad
  \ff_{7,2} &= -\tfrac{2^4}{4} \times \left( -\tfrac{1}{2} \right) = 2, \\
  \BF_{7,8}(x) &= -\tfrac{1}{6} x - \tfrac{7}{6} x^2 - \tfrac{7}{3} x^3
    + \tfrac{7}{3} x^5 + \tfrac{7}{6} x^6 + \tfrac{1}{6} x^7, &\quad
  \ff_{7,3} &= +\tfrac{2^5}{5} \times 0 = 0.
\end{alignat*}

Since the coefficients $\ff_{7,k}$ vanish for $k \geq 3$ due to the symmetry of
the antipalindromic polynomials $\BF_{7,2(k+1)}(x)$, this yields the
Faulhaber polynomial (see Table~\ref{tbl:F-poly})
\[
  F_7(y) = 2 y^2 - \tfrac{4}{3} y + \tfrac{1}{3}.
\]

Alternatively, Theorem~\ref{thm:BN-FP-poly} provides the polynomial
(see Table~\ref{tbl:FP-poly-odd})
\[
  \FP_7(u) = \tfrac{1}{2} u^3 - \tfrac{1}{2} u^2 + \tfrac{1}{6} u.
\]
Expanding $u = x(x-1)$ results in
$\tfrac{1}{2} x^6 -\tfrac{3}{2} x^5 + x^4 + \tfrac{1}{2} x^3 - \tfrac{1}{3} x^2 - \tfrac{1}{6} x$.
Multiplied by $2x-1$, one finally gets
\[
  \BN_7(x) = x^7 - \tfrac{7}{2} x^6 + \tfrac{7}{2} x^5 - \tfrac{7}{6} x^3 + \tfrac{1}{6} x.
\]
\end{example}
\vfill

\begin{example} \label{exp:coeff}
Computation of $\ff_{n,k}$ for $n=13$ and $k=4$.
By Corollary~\ref{cor:coeff-deriv} we have
\[
  \ff_{13,4} = - \tfrac{2^6}{6!} \times \BC_{13,10}^{(5)}(1).
\]
We can apply three different methods to compute the value of $\BC_{13,10}^{(5)}(1)$:
\begin{align*}
\shortintertext{\indent (1)~We can use Theorem~\ref{thm:recip-deriv} to compute that}
  \BC_{13,10}^{(5)}(x) &=
    -12 \, (4146 x^{12} - 4550 x^{10} + 429 x^8 + 130 x^2 - 390 x + 210)/x^8, \\
  \BC_{13,10}^{(5)}(1) &= 300. \displaybreak \\
\shortintertext{\indent (2)~We have by Theorem~\ref{thm:coeff-deriv} that}
  \BC_{13,10}^{(5)}(1) &= -5! \sum_{\nu=0}^{5} \binom{10-\nu}{5} \binom{13}{\nu} \BN_{13-\nu} \\
    &= -120 \, (1287 \, \BN_8 + 6006 \, \BN_{10} + 1638 \, \BN_{12}) \\
    &= -120 \mleft( 1287 \times \mleft( -\tfrac{1}{30} \mright)
       + 6006 \times \tfrac{5}{66}
       + 1638 \times \mleft( -\tfrac{691}{2730} \mright) \mright) \\
    &= 300. \\
\shortintertext{\indent (3)~It follows from Theorems~\ref{thm:coeff-deriv-2} and~\ref{thm:BF-reflect} that}
  \BC_{13,10}^{(5)}(1) &= 5! \sum_{\nu=0}^{5} \binom{-3}{5-\nu} \BF_{13,\nu} \\
    &= 120 \, (15 \, \BF_{13,1} - 10 \, \BF_{13,2} + 6 \, \BF_{13,3} - 3 \, \BF_{13,4} + \BF_{13,5}) \\
    &= 120 \mleft( 15 \times \tfrac{691}{210} - 10 \times \tfrac{1382}{35}
       + 6 \times \tfrac{20528}{105} - 3 \times \tfrac{10652}{21} + \tfrac{24384}{35} \mright) \\
    &= 300.
\end{align*}
Finally, this yields
$\ff_{13,4} = - \tfrac{2^6}{6!} \times 300 = - \tfrac{80}{3}$,
which matches the corresponding coefficient $-2800/105$ in Table~\ref{tbl:F-poly}.
\end{example}
\vfill

\setcounter{figure}{\value{theorem}}

\begin{figure}[H]
\begin{center}
\includegraphics[width=13.5cm]{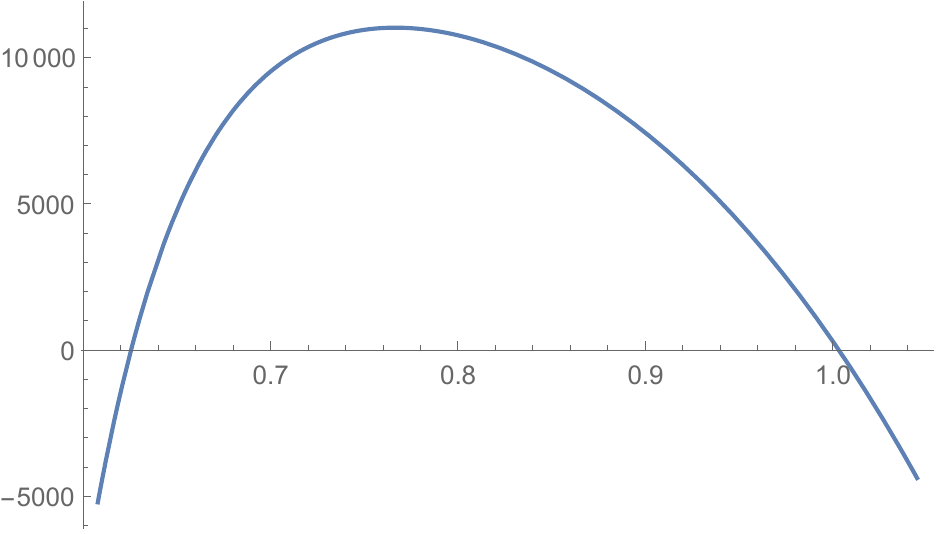}

\vspace*{3pt}
\begin{minipage}{11.0cm}
\begin{center}
\footnotesize
$\BC_{13,10}^{(5)}(x) = -12 \, \mleft( 4146 x^{12} - 4550 x^{10} + 429 x^8
+ 130 x^2 - 390 x + 210 \mright)/x^8$. \\
The zero near $x=1$ is located at $x \approx 1.003198$.
\end{center}
\end{minipage}

\caption{Graph of $\BC_{13,10}^{(5)}$.}
\label{fig:graph}
\end{center}
\end{figure}


\section*{Acknowledgments}

We would like to thank the anonymous referee for valuable comments and useful
suggestions, which improved the quality of the paper.


\newpage
\bibliographystyle{amsplain}

\begin{thebibliography}{10}
\setlength{\itemsep}{4pt}

\bibitem{ANFSSW:2003}
H.-W.~Alten, A.~D.~Naini, M.~Folkerts, H.~Schlosser, K.-H.~Schlote, and H.~Wu{\ss}ing,
\newblock \emph{4000 Jahre Algebra},
\newblock Springer--Verlag, Berlin, 2003.

\bibitem{Appell:1880}
P.~Appell,
\newblock \emph{Sur une classe de polyn\^omes},
\newblock Ann. Sci. \'Ecole Norm. Sup.~(2) \textbf{9} (1880), 119--144.

\bibitem{Bachmann:1910}
P.~Bachmann,
\newblock \emph{Niedere Zahlentheorie},
\newblock Part \textbf{2}, Teubner, Leipzig, 1910.
\newblock (Parts 1 and 2 reprinted in one volume, Chelsea, New York, 1968.)

\bibitem{Beardon:1996}
A.~F.~Beardon,
\newblock \emph{Sums of powers of integers},
\newblock Amer. Math. Monthly \textbf{103} (1996), 201--213.

\bibitem{Bernoulli:1713}
J.~Bernoulli,
\newblock \emph{Ars Conjectandi},
\newblock Basel, 1713.

\bibitem{Blij:1958}
F.~van~der~Blij,
\newblock \emph{Once again $\sum_{k=1}^{n-1} k^p$},
\newblock Euclides \textbf{34} (1958), 26--27. (Dutch)

\bibitem{Bruno:1857}
F.~F.~di~Bruno,
\newblock \emph{Note sur une nouvelle formule de calcul diff\'erentiel},
\newblock Quart. J. Pure Appl. Math. \textbf{1} (1857), 359--360.

\bibitem{Cajori:1928}
F.~Cajori,
\newblock \emph{A history of mathematical notations},
\newblock vol.~\textbf{1}, Open Court Publish. Comp., Chicago, 1928.

\bibitem{Cantor:1880}
M.~Cantor,
\newblock \emph{Vorlesungen \"uber Geschichte der Mathematik},
\newblock vol.~\textbf{1--2}, Teubner, Leipzig, 1880, 1892.

\bibitem{Carlitz:1962}
L.~Carlitz,
\newblock \emph{A note on sums of powers of integers},
\newblock Amer. Math. Monthly \textbf{69} (1962), 290--291.

\bibitem{Carlitz:1974}
L.~Carlitz,
\newblock \emph{A note on Bernoulli numbers and polynomials},
\newblock Elem. Math. \textbf{29} (1974), 90--92.

\bibitem{Comtet:1974}
L.~Comtet,
\newblock \emph{Advanced Combinatorics},
\newblock D. Reidel, Dordrecht, 1974.

\bibitem{Dirichlet:1856}
P.~G.~L.~Dirichlet,
\newblock \emph{Ged\"achtnisrede auf Carl Gustav Jacob Jacobi},
\newblock J. Reine Angew. Math. \textbf{52} (1856), 193--217.

\bibitem{Edwards:1982}
A.~W.~F.~Edwards,
\newblock \emph{Sums of powers of integers: a little of the history},
\newblock Math. Gaz. \textbf{66} (1982), 22--28.

\bibitem{Edwards:1986}
A.~W.~F.~Edwards,
\newblock \emph{A quick route to sums of powers},
\newblock Amer. Math. Monthly \textbf{93} (1986), 451--455.

\bibitem{Ettingshausen:1827}
A.~von~Ettingshausen,
\newblock \emph{Vorlesungen \"uber die h\"ohere Mathematik},
\newblock vol. \textbf{1}, Carl Gerold, Vienna, 1827.

\bibitem{Euler:1740}
L.~Euler,
\newblock \emph{De summis serierum reciprocarum},
\newblock Comm. Acad. Sci. Petrop.~\textbf{7} (1740), 123--134.

\bibitem{Euler:1741}
L.~Euler,
\newblock \emph{Inventio summae cuiusque seriei ex dato termino generali},
\newblock Comm. Acad. Sci. Petrop.~\textbf{8} (1741), 9--22.

\bibitem{Euler:1755}
L.~Euler,
\newblock \emph{Institutiones calculi differentialis cum eius usu in analysi finitorum ac doctrina serierum},
\newblock Acad. Imp. Sci. Petrop.~\textbf{1}, 1755.

\bibitem{Euler:1770}
L.~Euler,
\newblock \emph{De summis serierum numeros Bernoullianos involventium},
\newblock Novi Comm. Acad. Sci. Petrop.~\textbf{14} (1770), 129--167.

\bibitem{Faulhaber:1614}
J.~Faulhaber,
\newblock \emph{Newer Arithmetischer Wegweyser},
\newblock Johann Meder, Ulm, 1614; 2nd ed. 1617.

\bibitem{Faulhaber:1617}
J.~Faulhaber,
\newblock \emph{Continuatio seiner neuen Wunderk\"unsten},
\newblock Conradt Holtzhalb, N\"urnberg, 1617.

\bibitem{Faulhaber:1631}
J.~Faulhaber,
\newblock \emph{Academia Algebrae}, Johann Remmelin, Augsburg, 1631.

\bibitem{Gauss&Schumacher:1863}
C.~F.~Gauss and H.~C.~Schumacher,
\newblock \emph{Briefwechsel zwischen C.\,F.\,Gauss und H.\,C.\,Schumacher},
\newblock vol.~\textbf{5}, Gustav Esch, Altona, 1863.

\bibitem{Genocchi:1852}
A.~Genocchi,
\newblock \emph{Intorno all'espressione generale de'numeri Bernulliani},
\newblock Ann. Sci. Mat. Fis. \textbf{3} (1852), 395--405.

\bibitem{Gessel:2003}
I.~M.~Gessel,
\newblock \emph{Applications of the classical umbral calculus},
\newblock Algebra Univers. \textbf{49} (2003), 397--434.

\bibitem{Gessel&Viennot:1989}
I.~M.~Gessel and X.~G.~Viennot,
\newblock \emph{Determinants, paths, and plane partitions},
\newblock \href{http://people.brandeis.edu/~gessel/homepage/papers/pp.pdf}{preprint} (1989), 1--36.

\bibitem{Glaisher:1899}
J.~W.~L.~Glaisher,
\newblock \emph{On the sums of the series $1^n+2^n+\cdots+x^n$ and $1^n-2^n+\cdots\pm x^n$},
\newblock Quart. J. Pure Appl. Math. \textbf{30} (1899), 166--204.

\bibitem{Graham&others:1994}
R.~L.~Graham, D.~E.~Knuth, and O.~Patashnik,
\newblock \emph{Concrete Mathematics},
\newblock 2nd ed., Addison-Wesley, Reading, MA, 1994.

\bibitem{Hoppe:1846}
R.~Hoppe,
\newblock \emph{Ueber independente Darstellung der h\"ohern Differentialquotienten und den Gebrauch des Summenzeichens},
\newblock J. Reine Angew. Math. \textbf{33} (1846), 78--89.

\bibitem{Jacobi:1834}
C.~G.~J.~Jacobi,
\newblock \emph{De usu legitimo formulae summatoriae Maclaurinianae},
\newblock J.~Reine Angew.~Math. \textbf{12} (1834), 263--272.

\bibitem{Joffe:1915}
S.~A.~Joffe,
\newblock \emph{Sums of like powers of natural numbers},
\newblock Quart. J. Pure Appl. Math. \textbf{46} (1915), 33--51.

\bibitem{Johnson:2002}
W.~P.~Johnson,
\newblock \emph{The curious history of Fa\`a di Bruno's formula},
\newblock Amer. Math. Monthly \textbf{109} (2002), 217--234.

\bibitem{Kaestner:1799}
A.~G.~K\"astner,
\newblock \emph{Geschichte der Mathematik},
\newblock vol.~\textbf{3}, Rosenbusch, G\"ottingen, 1799.

\bibitem{Kellner:2021}
B.~C.~Kellner,
\newblock \emph{On (self-) reciprocal {Appell} polynomials: symmetry and Faulhaber-type polynomials},
\newblock Integers \textbf{21} (2021), \#A119, 1--19.

\bibitem{Knuth:1993}
D.~E.~Knuth,
\newblock \emph{Johann Faulhaber and sums of powers},
\newblock Math. Comp. \textbf{61} (1993), 277--294.

\bibitem{Lah:1954}
I.~Lah,
\newblock \emph{A new kind of numbers and its application in the actuarial mathematics},
\newblock Inst. Actu\'arios Portug., Bol. \textbf{9} (1954), 7--15.

\bibitem{Lampe:1878}
E.~Lampe,
\newblock \emph{Auszug eines Schreibens an Herrn Stern \"uber die ``Verallgemeinerung einer Jacobi'schen Formel.''},
\newblock J. Reine Angew. Math. \textbf{84} (1878), 270--272.

\bibitem{Lucas:1891}
\'{E}.~Lucas,
\newblock \emph{Th\'{e}orie des Nombres},
\newblock Gauthier--Villars, Paris, 1891.

\bibitem{Maclaurin:1742}
C.~Maclaurin,
\newblock \emph{A Treatise of Fluxions}, in two books, Ruddimans, Edinburgh, 1742.

\bibitem{Norlund:1922}
N.~E.~N{\o}rlund,
\newblock \emph{M\'emoire sur les polynomes de Bernoulli},
\newblock Acta Math. \textbf{43} (1922), 121--196.

\bibitem{Norlund:1924}
N.~E.~N{\o}rlund,
\newblock \emph{Vorlesungen \"uber Differenzenrechnung},
\newblock J. Springer, Berlin, 1924.

\bibitem{MacTutor:Jacobi}
J.~J.~O'Connor and E.~F.~Robertson,
\newblock \href{https://mathshistory.st-andrews.ac.uk/Biographies/Jacobi/}{\emph{Carl Gustav Jacob Jacobi}},
\newblock MacTutor History of Mathematics.

\bibitem{Pacioli:1494}
L.~Pacioli,
\newblock \emph{Summa de arithmetica geometria},
\newblock Paganino Paganini, Venice, 1494.

\bibitem{Poisson:1826}
S.~D.~Poisson,
\newblock \emph{M\'emoire sur le calcul num\'erique des int\'egrales d\'efinies},
\newblock Acad. Sci. \textbf{6} (1826), 1--34.

\bibitem{Prasolov:2010}
V.~V.~Prasolov,
\newblock \emph{Polynomials},
\newblock 2nd ed., ACM \textbf{11}, Springer--Verlag, Berlin, 2010.

\bibitem{Raabe:1848}
J.~L.~Raabe,
\newblock \emph{Die Jacob Bernoullische Function},
\newblock Orell, F\"ussli \& Co., Z\"urich, 1848, 1--51.

\bibitem{Raabe:1851}
J.~L.~Raabe,
\newblock \emph{Zur\"uckf\"uhrung einiger Summen und bestimmten Integrale auf die Jacob-Bernoullische Function},
\newblock J. Reine Angew. Math. \textbf{42} (1851), 348--367.

\bibitem{Schneider:1983}
I.~Schneider,
\newblock \emph{Potenzsummenformeln im 17. Jahrhundert},
\newblock Hist. Math. \textbf{10} (1983), 286--296.

\bibitem{Schneider:1993}
I.~Schneider,
\newblock \emph{Johannes Faulhaber 1580--1635, Rechenmeister in einer Welt des Umbruchs},
\newblock  Vita Mathe\-matica \textbf{7}, Birkh\"auser--Verlag, Basel, 1993.

\bibitem{Schroeder:1867}
E.~Schr\"oder,
\newblock \emph{Eine Verallgemeinerung der Mac-Laurinschen Summenformel},
\newblock Kantonsschule, Z\"urich, 1867, 1--28.

\bibitem{Stern:1878}
M.~Stern,
\newblock \emph{Verallgemeinerung einer Jacobischen Formel},
\newblock J. Reine Angew. Math. \textbf{84} (1878), 216--219.

\bibitem{Sturm:1859}
C.~Sturm,
\newblock \emph{Cours d'Analyse},
\newblock  vol.~\textbf{2}, Mallet--Bachelier, Paris, 1859.

\bibitem{Tits:1922}
L.~Tits,
\newblock \emph{Identit\'es nouvelles pour le calcul des nombres de Bernoulli},
\newblock Nouv. Ann. Math. \textbf{1} (1922), 191--196.

\end{thebibliography}

\end{document}